\documentclass{amsart}
\usepackage{amsmath}
\usepackage{amssymb}
\usepackage{amsthm}
\usepackage{graphicx}
\usepackage{bm} 
\usepackage{mathtools}

\newtheorem{theorem}{Theorem}

\theoremstyle{definition}

\theoremstyle{remark}
\newtheorem{remark}{Remark}

\usepackage{mathrsfs}
\usepackage{subcaption}

\usepackage{doi}
\usepackage[textsize=scriptsize]{todonotes}
\usepackage{algorithmicx}
\usepackage{algpseudocode}
\usepackage{algorithm}
\usepackage{xcolor}

\newcommand{\ave}[1]{{#1}^{ave}}
\newcommand{\vex}[1]{{#1}^{vex}}

\allowdisplaybreaks

\begin{document}

\title[Reliability optimization using convex envelopes]{Exact reliability optimization for series-parallel graphs using convex envelopes}
\thanks{The research leading to these results received funding from grants ANID Fondecyt Regular 1200809 (J.B., E.M.); ANID Fondecyt Iniciaci\'on 11190515 (G.M.); ANID ANILLO ACT192094 (E.M., J.B.); Program Math Amsud 19-MATH-03 (J.B.,P.R.)}

\author{Javiera Barrera}
\address{Faculty of Engineering and Sciences,  Universidad Adolfo Ib\'a\~nez, Santiago, Chile}
\email{javiera.barrera@uai.cl}
\author{Eduardo Moreno}
\address{Faculty of Engineering and Sciences,  Universidad Adolfo Ib\'a\~nez, Santiago, Chile}
\email{eduardo.moreno@uai.cl} 
\author{Gonzalo Mu\~noz}
\address{Institute of Engineering Sciences, Universidad de O'Higgins,  Rancagua, Chile}
\email{gonzalo.munoz@uoh.cl} 
\author{Pablo Romero}
\address{Faculty of Engineering,  Universidad de la Rep\'ublica,  Montevideo, Uruguay}  
\address{Faculty of Natural and Exact Sciences,  Universidad de Buenos Aires, Argentina}
\email{promero@fing.edu.uy}

\begin{abstract}
Given its wide spectrum of applications, the classical problem of all-terminal network reliability evaluation remains a highly relevant problem in network design.

The associated optimization problem---to find a network with the best possible reliability under multiple constraints---presents an even more complex challenge, which has been  addressed in the scientific literature but usually under strong assumptions over failures probabilities and/or the network topology.

In this work, we propose a novel reliability optimization framework for network design with failures probabilities that are independent but not necessarily identical. We leverage the linear-time evaluation procedure for network reliability in the series-parallel graphs of Satyanarayana and Wood~\cite{Satya1985} to formulate the reliability optimization problem as a mixed-integer nonlinear optimization problem. To solve this nonconvex problem, we use classical convex envelopes of bilinear functions, introduce custom cutting planes, and propose a new family of convex envelopes for expressions that appear in the evaluation of network reliability. Furthermore, we exploit the refinements produced by spatial branch-and-bound to locally strengthen our convex relaxations.
Our experiments show that, using our framework, one can efficiently obtain optimal solutions in challenging instances of this problem.
\end{abstract}

\maketitle

\section{Introduction}\label{intro}
Network design to optimize the reliability of a network is one of the most classic problems in the optimization literature, 
with applications in many of the aspects of our lives and certainly in many more in the future. In the recent special issue to celebrate the 50-year anniversary of the \emph{Networks} journal, a beautiful and exhaustive review by Brown et al.~\cite{brown2021network} was published, which covers the different notions, theory and applications of network reliability, along with future directions. In a similar way, P\'erez-Ros\'es~\cite{SixtyYears} revisited the 60 years of this problem, initially started by E.F. Moore and C.E Shannon~\cite{moore1956reliable} in 1956. This paper proposes a new methodology to advance in one of the most classical and defying problems in this area, the exact reliability optimization of a network.

Consider a simple undirected graph $G=(V,E)$ representing a network whose edges fail according to a random probability---these failures are assumed to be independent. We consider one of the most classical notions of reliability: the \emph{all-terminal reliability} of $G$, which is defined as the probability that $G$ remains connected. 

The exact reliability evaluation belongs to the class of $\mathcal{NP}$-hard problems~\cite{Ball1983}, even when only two terminals are required to be connected~\cite{Valiant1979} or under identical link failure probabilities~\cite{Rosenthal}. However, for certain specific families of graphs, the problem can be solved in polynomial time. In the particular case of series-parallel graphs, Satyanarayana and Wood~\cite{Satya1985} provide a set of reliability-preserving reductions to simplify the graph and its structure that make it possible to compute the reliability of a series-parallel network in linear time. This work is our starting point. This type of reduction can also be applied to general graphs to reduce the size and complexity of the problem. For example, a recent work~\cite{goharshady2020efficient} extended these results to obtain a parametrized algorithm for computing the reliability of graphs with small treewidth in linear time. Nonetheless, this approach has not been thoroughly exploited in the literature.

% A nice monograph on the combinatorics of network reliability was authored by Colbourn~\cite{Colbourn87}. Given the hardness of the reliability evaluation, several pointwise estimation techniques are available, such as recursive variance reduction or RVR~\cite{CancelaKR12}, importance sampling or IS~\cite{IS-2015} and cross-entropy method~\cite{crossentropy}. The theory of exact reliability evaluation is already mature, and exponential-time factorization methods ~\cite{SatyaChang1983} as well as sum of disjoint products~\cite{SurveyTrivedi} methods are available, among others. See~\cite{SixtyYears} for a recent review on the network reliability problem and its extensions.

As Brown et al.~\cite{brown2021network} note in their recent review, one of the most relevant and less-traveled paths is the optimization of network reliability. That is, given a limited set of resources (e.g., number of edges), how should one select a subgraph of $G$ that maximizes its reliability. Given the hardness of computing the exact reliability of a network, using this metric as the objective function for an optimization problem can be highly impractical in general. However, since the computation of the reliability in certain families of graphs is tractable and due to the sharp progress we have witnessed in mixed-integer nonlinear optimization technology, an efficient method for reliability optimization for particular classes of graphs is plausible.

%Topological optimization problems meeting reliability constraints or reliability maximization under budget constraints -- known as network synthesis--  are also found  in the scientific literature. 
A survey on reliability optimization is provided in~\cite{Boesch2009}. Most previous works have assumed either identical link costs or independent link failures with identical probability; in these cases, the optimal topologies are highly symmetrical~\cite{zia2010redundancy}. A foundational work on this problem for the all-terminal reliability problem was carried out by Boesch~\cite{BoeschHello}. In this work, a fixed budget is considered under identical costs (i.e., the network has a precise number of links), and identical independent link failures with probability $\rho$ are assumed. The goal is to design a graph whose  all-terminal reliability is maximized in a uniform sense, i.e., over the whole compact set $\rho \in [0,1]$,  among all the graphs with the same number of nodes and links. These graphs are known as \emph{uniformly most-reliable graphs} or UMRGs. It is known that UMRGs must have the maximum tree number and maximum connectivity, providing evidence of the symmetry of the optimal graphs under the strong assumptions of independence and identical costs. Nevertheless, several historical conjectures regarding the construction and existence of UMRGs remain open~\cite{UMRG2019}. In the restricted class of series-parallel graphs, their UMRGs are characterized in~\cite{neufeld1985most}. 

One way of tackling the difficulties associated with exact reliability computation is to use meta-heuristics to provide a \emph{good} solution for the problem; this approach was highly studied during the 1990s~\cite{aboelfotoh2001neural,dengiz1997efficient,deeter1998economic}. Later work on this subject includes hybrid ACO~\cite{dengiz2010design} and self-tuning heuristics~\cite{dengiz2015self}. More recently, \cite{ozkan2020reliable} combines heuristic techniques with branch-and-bound methods. However, the main drawback of these techniques is that heuristics cannot provide guarantees on the optimality of  their solutions. A different approach to overcome this problem is to use simulations of the failure process and embed these sampled scenarios into an optimization model, known as the sample average approximation method~\cite{kleywegt2002sample}. This approach allows one to consider failure models without assuming independent or identical failures probabilities but generates an approximated formulation.  
For example, \cite{SAA} provides a powerful methodology to minimize the cost of a network meeting two-terminal reliability constraints, which considers probabilistic  cuts in the branch-and-bound tree. Similarly, \cite{TopologicalRel2015} presents a reliability optimization model with dependent link failures ruled by the Marshall-Olkin copula~\cite{MarshalOlkinParameters,MarshalOlkin2}. %, and their results suggest that suboptimal solutions could be found if dependent failures are ignored in the model. A parametric estimation of the Marshal-Olkin copula was recently presented in~\cite{MarshalOlkinParameters,MarshalOlkin2} that jointly with~\cite{TopologicalRel2015} showed a promising research direction toward the development of realistic models. 
In \cite{TOP}, a stochastic network flow model using scenarios for the two-terminal reliability case is provided.

Since the nonlinear functions representing the reliability of the problem are neither convex nor concave functions, straightforward mathematical optimization formulations seem to be of little use in reliability optimization problems. Nonetheless, one can rely on convexification techniques to approximate the problem with a tractable alternative and exploit spatial branch-and-bound to further refine such approximations. This is a key component of state-of-the-art mixed-integer nonlinear programming (MINLP) technology and, to the best of our knowledge, it has not been exploited for the general network reliability problem. For example, convex envelopes of functions can be efficiently computed under special cases, which yield strong convex relaxations. This idea is used in~\cite{ye2018mixed} to optimize the design of a reliable chemical plant, which can be represented as a series graph. Similar techniques have been used for  other network-related problems, for example, on AC optimal power flow problems~\cite{bynum2018strengthened}.

Our paper follows this previous idea. We consider the optimization reliability problem with independent failure but with different failure probability among edges. In this setting, the reduction techniques remain valid~\cite{Satya1985}.  
Based on these reliability-preserving serial-parallel reductions, a convex MINLP formulation for the reliability optimization problem is obtained for series-parallel graphs. In this formulation, each reduction generates new constraints (a linear number). We provide tight convex envelopes for the functions appearing from the reduction process, which, combined with the refinements carried out in spatial branch-and-bound trees, allows us to obtain the (exact) optimal solution efficiently. This idea could also be extended to other families of graphs: for instance, one can rely on series-parallel reductions to decrease the size of a problem.

This article is organized as follows. Section~\ref{sec:concepts} presents 
series and parallel reductions that preserve all-terminal reliability, following the work 
of Satyanarayana and Wood~\cite{Satya1985}. Section~\ref{sec:main} presents the main contributions of this work. 
Specifically, a nonlinear and nonconvex formulation of the reliability optimization problem is introduced in Subsection~\ref{basic}, considering series-parallel reductions. Subsection~\ref{sec:envelopes} introduces convex envelopes associated with the series-parallel reductions, using classical McCormick envelopes \cite{mccormick1976computability} and a novel envelope for series-type reductions (Theorem~\ref{thm:main}). 
%The resulting problem falls within the framework of convex optimization.
Further improvements to the resulting convex optimization problem are addressed in Subsection~\ref{improvements}. The computational effectiveness of our proposal is studied in Section~\ref{sec:results}. 
Finally, Section~\ref{sec:conclusions} presents concluding remarks and directions for future work.

\section{Definitions and reliability-preserving reductions}\label{sec:concepts}
Consider an undirected graph $G=(V,E)$. Nodes are perfectly reliable, but links may fail with independent probabilities $q_e$ for $e=1,\ldots |E|$. Let us 
denote by $p_e=1-q_e$ the elementary reliability of the link $e$. We denote by $\mathcal{R}_G(p)$ the \emph{all-terminal reliability} of graph $G$: the probability that $G$ is connected. 
%That is, the probability that for any pair of nodes $u,v\in V$, there exists a path connecting these nodes. 

Given two graphs $G_1$ and $G_2$ with two distinguished vertices $s(G_i)$ and $t(G_i)$, a \emph{series composition} of $G_1$ and $G_2$, denoted by $G_1 +_S G_2$, is the disjoint union of both graphs, merging $t(G_1)$ with $s(G_2)$. In this case, $s(G_1 +_S G_2) = s(G_1)$ and  $t(G_1 +_S G_2) = t(G_2)$. Similarly, a \emph{parallel composition}  $G_1 +_P G_2$ is the disjoint union of both graphs, merging  $s(G_1)$ with $s(G_2)$ (thus $s(G_1) = s(G_2) = s(G_1 +_S G_2) $ and merging $t(G_1)$ with $t(G_2)$ (thus $t(G_1)=t(G_2)=t(G_1 +_P G_2)$). 

A graph $G$ is a \emph{series-parallel graph} if it can be obtained from a sequence of series-parallel compositions starting from the single edges $G_e=\{e\}$ for $e= 1\ldots |E|$. Formally, let $\mathcal{G}_0=\{ G_e : e=1\ldots |E|\}$ be the set of single edges. Iteratively, we construct $\mathcal{G}_i = \mathcal{G}_{i-1} \bigcup (G_{|E|+i} \setminus (G_j \cup G_k))$ where $G_{|E|+i}=G_j+_\odot G_k$ is either a series or parallel composition of graphs $G_j$ and $G_k$ in $\mathcal{G}_{i-1}$. Note that the number of connected components in $\mathcal{G}_i$ is $|E|-i$ because this number decreases by one in each iteration. Therefore, $\mathcal{G}_{|E|-1}$ contains only one element, which is a connected series-parallel graph.  
We denote by $\mathcal{S}_G = \left[G_{|E|+i}=G_j+_\odot G_k\right]_{i=1}^{|E|-1}$ this sequence of series-parallel compositions to construct $G$. Note that this sequence is not unique for a given graph $G$.

%A graph $G$ is a \emph{series-parallel graph} if it can be obtained from a sequence of series-parallel compositions $\mathcal{S}_G = \left[G_i=G_j+_\odot G_k\right]_{i=1}^{|E|}$ starting from the single edges $G_e=\{e\}$ for $e= 1\ldots |E|$ and replacing subgraphs $G_j$ and $G_k$ by the new constructed subgraph $G_i$ at each step.  If $G$ is connected, then $\mathcal{S}_G$ has exactly $|E|-1$ series or parallel compositions. Note that this sequence of composition is not unique for a given graph $G$.

Satyanarayana and Wood~\cite{Satya1985} presented a set of reliability-preserving transformations for computing the reliability of series-parallel graphs based on a sequence of series-parallel compositions.%\todo{GM:Acá (ver siguiente todo)}
These transformations based on series-parallel compositions can also be used in general graphs to reduce their size in their reliability computation.

When two edges $e_j,e_k$  are \emph{in parallel}, with reliabilities $p_j,p_k$, these edges can be replaced by a new single edge $e_i$ with reliability $p_i = 1-(1-p_j)(1-p_k)$, which is the probability that at most one of these two links  fails.

If two edges $e_j,e_k$ are \emph{in series}, with reliabilities $p_j,p_k$, at least one of them must remain operational to keep the graph connected. Thus, if we replace these two edges with a new edge $e_i$, the reliability of this edge must consider this event. Let $\mathcal{A}$ be the event that $e_j$ and $e_j$ do not fail simultaneously. Hence, the reliability of the graph satisfies
\begin{align*}
\mathcal{R}_G &= \mathbb{P}[G \text{ is connected} | \mathcal{A}]\cdot \mathbb{P}[\mathcal{A}] + \underbrace{\mathbb{P}[G \text{ is connected} | \mathcal{A}^C]}_{0} \cdot \mathbb{P}[\mathcal{A}^C] \\
&= \mathbb{P}[G \text{ is connected} | \mathcal{A}]\cdot \mathbb{P}[\mathcal{A}] %=  \mathcal{R}_{G^{\prime}}\cdot \Omega_i ,  
\end{align*}
Therefore, in the case of a series reduction replacing edges $e_j$ and $e_k$ with a new edge $e_i$, the reliability of the resulting graph must be multiplied by $\mathbb{P}[\mathcal{A}] = 1-(1-p_j)(1-p_k)$, and the reliability of the new edge $e_i$ is the probability that both edges are operational conditional to the event that at least one of them remains operational, that is $p_j \cdot p_k$ normalized by the probability of $\mathcal{A}$:
$$p_i = \frac{p_j p_k}{1-(1-p_j)(1-p_k)}.$$

% Hence, we normalize the final reliability by this event. Therefore, we can replace these edges by a new edge $e_i$ with reliability 
% %$$p_i = p_j p_k / (1-(1-p_j)(1-p_k))$$
% $$p_i = \frac{p_j p_k}{1-(1-p_j)(1-p_k)}$$
% That is, the probability that both edges $e_j$ and $e_k$ do not fail ($p_j\cdot p_k$) divided by the probability that at least one of them does not fail ($1-(1-p_j)(1-p_k)$). Additionally to this last transformation, we need to apply a normalization factor to the resulting graph reliability, as $p_i$ is a conditional probability.
%
% If we let $G^{\prime}$ be the graph obtained after replacing $e_j$ and $e_k$ with $e_i$ as above, its reliability $\mathcal{R}_{G^{\prime}}$ must be multiplied by $\Omega_i=1-(1-p_j)(1-p_k)$. This is explained because the reliability $R_G$ can be conditioned by the event $\mathcal{A}$ such that $p_j$ and $p_k$ do not fail simultaneously to obtain:
% \begin{align*}
% \mathcal{R}_G &= \mathbb{P}[G \text{ is connected} | \mathcal{A}]\cdot \mathbb{P}[\mathcal{A}] + \underbrace{\mathbb{P}[G \text{ is connected} | \mathcal{A}^C]}_{0} \cdot \mathbb{P}[\mathcal{A}^C] \\
% &= \mathbb{P}[G \text{ is connected} | \mathcal{A}]\cdot \mathbb{P}[\mathcal{A}] =  \mathcal{R}_{G^{\prime}}\cdot \Omega_i ,  
% \end{align*}
% where it is used the fact that the resulting graph is never connected under the event $\mathcal{A}^C$, 
% and $\mathbb{P}[\mathcal{A}] = \Omega_i$.

Finally, note that the sequence $\mathcal{S}_G$ that \emph{constructs} the graph $G$ can be used to compute the reliability of $G$. That is, if $G_{|E|+i}=G_j+_\odot G_k$ and $G_j$ and $G_k$ corresponds to an edge of $E$, then $G_{|E|+i}$ is a new edge that replaces the original two edges with a new edge representing its series/parallel composition. Therefore, applying this sequence iteratively, at each step of the sequence $\mathcal{S}_G$, the subgraphs $G_j,G_k$ in a composition are edges. %\todo{GM: No les parece mover este párrafo completo más arriba? Partir diciendo que la reducción se aplica a la secuencia que se usa para crear $G$. Lo pondría donde dejé el todo arriba.}

We formalize the reliability computation of a series-parallel graph $G$ in Algorithm~\ref{alg:computeReliability}.
\begin{algorithm}
\caption{Compute the reliability of a series-parallel graph $G$ from its composition sequence $\mathcal{S}_G$.} \label{alg:computeReliability}
\begin{algorithmic}
\Require Composition sequence $\mathcal{S}_G = \left[G_{|E|+i}=G_j+_\odot G_k\right]_{i=1}^{|E|-1}$
\Require $p_e$ reliabilities of edges $e\in 1\ldots |E|$
\For {$i=1\ldots |E|$}
\State $G_i \gets \{i\}$
\State $Y_i \gets p_i$
\State $\Omega_i \gets 1$
\EndFor
\For {$i=1\ldots |E|-1$}
\If{$G_{|E|+i}=G_j+_P G_k$} 
\State $Y_{|E|+i} \gets 1-(1-Y_j)(1-Y_k)$
\State $\Omega_{|E|+i} \gets 1$
\ElsIf{$G_{|E|+i}=G_j+_S G_k$}
\State $Y_{|E|+i} \gets \frac{Y_j Y_k}{1-(1-Y_j)(1-Y_k)} $
\State $\Omega_{|E|+i} \gets 1-(1-Y_j)(1-Y_k)$
\EndIf
\EndFor
\State \Return $\mathcal{R}_G := Y_{2|E|-1} \cdot \prod_{i=1}^{2|E|-1} \Omega_i$
\end{algorithmic}
\end{algorithm}
In other words, $Y_i$ represents the reliability of the edge $i$ for $i\leq |E|$, or the reliability of the edge resulting from the reduction $G_{|E|+i}=G_j+_\odot G_k$ for $i> |E|$. Similarly, $\Omega_{|E|+i}$ represents the reliability factor from the reduction $G_{|E|+i}$, which is either $\Omega_{|E|+i}=1$ (parallel composition) or $\Omega_{|E|+i}=1-(1-p_j)(1-p_k)$ (series composition). When the graph has been reduced to a single edge, the reliability of $G$ is equal to the reliability of this edge ($Y_{2|E|-1}$) multiplied by all the factors $\Omega_i$.  Algorithm~\ref{alg:computeReliability} allows one to compute the all-terminal reliability of $G$ in linear time.
%Therefore, if $G$ is a series-parallel graph, by applying these reductions over the sequence of series-parallel compositions $\mathcal{S}_G$, we can compute the all-terminal reliability of $G$ in linear time.

\section{Optimizing the reliability of a series-parallel graph}\label{sec:main}
\subsection{A nonlinear optimization model for a series-parallel graph}\label{basic}
The aforementioned results provide a procedure to compute the resulting reliability $\mathcal{R}_G$ of a series-parallel graph $G$. Our main interest is in studying the network design problem of selecting the subgraph that maximizes reliability given a set of constraints.
Specifically, given a graph $G=(V,E)$, we are interested in the selection of a subset of edges $F\subseteq E$ satisfying the given constraints such that the reliability of the graph $(V,F)$ is maximized. 

We now proceed to formulate this problem as an MINLP. Let $X_e\in \{0,1\}$ for $e\in E$ be binary variables indicating whether $e\in F$, and let $A X \leq b$ be a given set of the arbitrary linear constraints that any valid $X$ must satisfy. These can be, for example, an upper bound on the number of edges to be considered.
Let $\mathcal{R}(p)$ be the reliability of network $G$ given the probability vector $p\in[0,1]^E$---each component $p_e$ is the reliability of $e\in E$. We note that if a link $e\in E$ is not considered in $F$, this is equivalent to assuming that its elementary reliability is 0. Then, we can formulate our problem as:
\begin{align*}
\max &\ \mathcal{R}(p_1X_1,p_2X_2,\ldots ,p_{|E|} X_{|E|})\\
s.t. \, \, & A X \leq b \\
& X_{e} \in \{0,1\} \quad \forall e\in 1\ldots |E|
\end{align*}

If $G$ is a series-parallel graph, we can apply the reliability-preserving reductions over the composition sequence $\mathcal{S}_G = \left[G_{|E|+i}=G_j+_\odot G_k\right]_{i=1}^{|E|-1}$ that constructs $G$. Following the idea and notation behind Algorithm~\ref{alg:computeReliability}, we define the continuous variables $Y_i\in[0,1]$ and  $\Omega_i\in[0,1]$ for each $i=1\ldots 2|E|-1$ to represent the reliability and the correction factor of each step of the sequence. Additionally, we define the continuous variables $\bar\Omega_i\in[0,1]$ to represent the product of the correction factors. Using these variables, we can formulate the problem as follows:

\begin{subequations}\label{eq:mainmodel}
\begin{align}
	\max\ & R &\label{eq:original_start}\\
 A X &\leq b &  \label{eq:sideConstraints}\\
 Y_i &= p_i X_i &  & & i=1\ldots |E| \label{eq1:original}\\
\Omega_i &= 1 & & & i=1\ldots |E| \label{eq1:original2}\\
 Y_{|E|+i} &= 1-(1-Y_j)(1-Y_k) &  & &    i=1\ldots |E|-1 : G_{|E|+i} = G_j +_P G_k\label{eq1:parallel}\\
 \Omega_{|E|+i} &= 1 &  & &  i=1\ldots |E|-1 : G_{|E|+i} = G_j +_P G_k\label{eq1:parallel2}\\
 Y_{|E|+i} &= \frac{Y_j Y_k}{1-(1-Y_j)(1-Y_k)} & & &   i=1\ldots |E|-1 : G_{|E|+i} = G_j +_S G_k \label{eq1:series}\\
 \Omega_{|E|+i} &= 1-(1-Y_j)(1-Y_k) & & &   i=1\ldots |E|-1 : G_{|E|+i} = G_j +_S G_k \label{eq1:series2}\\
 \bar{\Omega}_1 &= \Omega_1 & &  &  i=2\ldots 2|E|-1 \label{eq1:cumulative}\\
 \bar{\Omega}_i &= \bar{\Omega}_{i-1} \cdot \Omega_i & & & i=2\ldots 2|E|-1 \label{eq1:cumulative2}\\
 R &= Y_{2|E|-1} \cdot \bar{\Omega}_{2|E|-1} & \label{eq1:objetive}\\
 Y_i, \Omega_i, \bar{\Omega}_i&\in[0,1]  & & &  i=1\ldots 2|E|-1\\ 
 X_i &\in \{0,1\}  && & i=1\ldots |E|\label{eq:original_finish}
\end{align}
\end{subequations}

Constraints~\eqref{eq1:original}-\eqref{eq1:original2}  correspond to the main decision variables, indicating whether an edge $e$ is considered in the subgraph $F$. In the latter case, edge $e$ has elementary reliability equal to zero. Constraints~\eqref{eq1:parallel}-\eqref{eq1:parallel2} model a parallel reduction $G_{|E|+i} = G_j +_P G_k$, in which case the new edge $i$ has reliability $1-(1-p_i)(1-p_k)$ and there is no reliability correction factor ($\Omega_{|E|+i}=1$). Constraints~\eqref{eq1:series}-\eqref{eq1:series2} model a series reduction, where the new edge has reliability $\frac{p_j p_k}{1-(1-p_j)(1-p_k)}$ and the reliability correction factor is $\Omega_{|E|+i}=1-(1-p_j)(1-p_k)$. Constraints~\eqref{eq1:cumulative}-\eqref{eq1:cumulative2} ensure that $\bar{\Omega}_i$ is equal to the cumulative product of the factors $\Omega_i$, that is, $\bar{\Omega}_i = \prod_{j<i} \Omega_j$. Finally, constraint~\eqref{eq1:objetive} provides the reliability of the graph $R$, which is the operational probability after the last reduction, i.e., when $G$ has been reduced to a single edge. 

Note that this problem can be seen as a \emph{mixed-integer quadratically constrained program} (MIQCP); the left side of constraint~\eqref{eq1:series} can also be written as $Y_i\cdot \Omega_i = Y_j\cdot  Y_k$. However, all quadratic constraints are nonconvex, and thus this model can be challenging for most of the current nonlinear optimization solvers.

\subsection{Convex envelopes for the problem}\label{sec:envelopes}

The main issue with model \eqref{eq:mainmodel} is that the resulting constraints involve nonconvex and nonconcave functions. In fact, the three bivariate functions $f_1(x,y)=xy$, $f_2(x,y) = 1-(1-x)(1-y)$ and $f_3(x,y)=\tfrac{xy}{1-(1-x)(1-y)}$ appearing in \eqref{eq:mainmodel} are neither convex nor concave functions\footnote{Without loss of generality, we assume that $f_3(0,0) = \lim_{(x,y)\rightarrow (0,0)} f_3(x,y) = 0$.}. 
%Hence, this model is not suitable for most of the non-linear optimization solvers. 

One common approach for generating a (possibly strong) convex relaxation is to use the \emph{envelopes} of these functions. The concave envelope of $f(x,y)$ over a given domain $D$ is the smallest concave overestimator $\ave{f}(x,y)\geq f(x,y)$ for all $(x,y)\in D$ and can be used to relax a constraint of type $f(x,y)\geq z$ with a convex constraint $\ave{f}(x,y)\geq z$. Similarly, the convex envelope of $f(x,y)$ is the largest convex underestimator $\vex{f}(x,y)\leq f(x,y)$ for all $(x,y)\in D$ and can be used to relax a constraint $f(x,y)\leq z$. 
In our setting, this implies that we can relax an equality constraint $z=f_i(x,y)$ with two convex constraints $ \vex{f}_i(x,y) \leq z \leq \ave{f}_i(x,y)$.

Since our optimization problem only considers equality constraints, in principle, we should aim at computing both convex and concave envelopes. However, due to the structure of our problem, constraints $\vex{f}_i(x,y) \leq z$ are not necessary.
We show below that $\ave{f}_i(x,y)$ are all increasing functions in both variables, and considering that we are maximizing reliability, along with the simple structure of our constraints, $z \leq \ave{f}_i(x,y)$ will always be active in an optimal solution of the resulting convex relaxation. This is expected, as $f_1(x,y), f_2(x,y)$ and $f_3(x,y)$ are increasing functions in both variables in the square $[0,1]\times[0,1]$.
% since we are maximizing the reliability, this implies that the convex envelopes do not play a significant role and thus we focus on the concave envelopes of these functions. 
% \todo[inline]{¿Decir tambien que $f^{AVE}$ es creciente?} 

For the case of $f_1(x,y)=x\cdot y$, its envelopes are well-known and can be obtained by the McCormick envelopes~\cite{mccormick1976computability}. Let us assume that $x,y\in [0,1]$, and let $L_x,L_y$ and $U_x,U_y$ be lower and upper bounds for $x$ and $y$; then, the concave envelope of $f_1(x,y)$ is given by:
\begin{equation}
	f_1(x,y) = x\cdot y \leq \begin{cases}
	U_x \cdot y + x\cdot L_y - U_x\cdot L_y & \text{if } y - L_y \geq \frac{U_y-L_y}{U_x-L_x} \cdot (x - L_x) \\
	x \cdot U_y + L_x\cdot y - L_x \cdot U_y & \text{if not} \end{cases}
\end{equation}
This envelope is increasing in both variables when the variable bounds are nonnegative. 

For the case of $f_2(x,y)=1-(1-x)(1-y) = x + y - x\cdot y$, its concave envelope can be obtained using the convex envelope of $f_1(x,y)$, resulting in the following piecewise linear function:
\begin{equation}
	f_2(x,y) = 1-(1-x)(1-y) \leq \begin{cases}
	x \cdot (1-L_y) + y\cdot (1-L_x) + L_x L_y & \text{if }U_y - y \geq \frac{U_y-L_y}{U_x-L_x} \cdot (x - L_x)\\
	x \cdot (1-U_y) + y\cdot (1-U_x) + U_x U_y & \text{if not} \end{cases}
\end{equation}
Since in our case both variable bounds are less than 1, this envelope is also increasing in both variables.

For the case of $f_3(x,y)=\tfrac{x\cdot y}{x+y-xy}$, an explicit formula for its concave envelopes is not known. Here, we provide its concave envelope for the case of $L_x=L_y=0$.

\begin{theorem}\label{thm:main}
The concave envelope of $f_3(x,y)=\tfrac{x\cdot y}{x+y-xy}$ for $(x,y)\in[0, U_x]\times [0, U_y]$ for $U_x,U_y\leq 1$ is given by:
\begin{equation}
	\ave{f}_3(x,y) = \begin{cases} 
\frac{x\cdot y}{x+y-U_x\cdot y} & \text{if } x/U_x \geq y/U_y \\
\frac{x\cdot y}{x+y-x\cdot U_y} & \text{if not} \end{cases}
\end{equation}
\end{theorem}
\begin{proof}
It is easy to see that $\ave{f}_3(x,y) \geq f_3(x,y)$. On the other hand, note that for  $(x,y)\in [0,U_x]\times[0,U_y]$,
$$\ave{f}_3(x,y)= \min\left\{\frac{x\cdot y}{x+y- U_x y}, \frac{x\cdot y}{x+y-x U_y}\right\}$$
The Hessian matrix of  $\frac{x\cdot y}{x+y- U_x y}$ is: 
\[ \mathbf{H} \left(\frac{x\cdot y}{x+y-U_x\cdot y}\right) = \frac{2 (1-U_x)}{(x+y-U_x\cdot y)^3} \begin{bmatrix} -y^2 & x y \\ x y & -x^2 \end{bmatrix} \]
which is a negative semidefinite-matrix, and thus it is a concave function in $ [0,U_x]\times[0,U_y]$. The same result can be obtained for $\frac{x\cdot y}{x+y-x\cdot U_y}$ by exchanging $x$ and $y$.
%A similar proof can be obtained for $\frac{x\cdot y}{x+y-x\cdot U_y}$ by interchanging $x$ and $y$. 
Finally, since $\ave{f}_3(x,y)$ is the minimum of these two concave functions, we conclude that $\ave{f}_3(x,y)$ is concave over $[0,U_x]\times[0,U_y]$. 

Finally, we need to show that $\ave{f}_3(x,y)$ is the \emph{smallest} concave overestimator of $f_3(x,y)$. We note that if $y=\lambda x$, then the function $f_3(x,\lambda x) = \tfrac{\lambda x}{1+\lambda-\lambda x}$ is a convex function on $x$. In fact, $\tfrac{\partial f_3(x,\lambda x)}{\partial x} = 2\lambda^2(1+\lambda) / (1+\lambda-\lambda x)^3$, which is positive for any $\lambda>0$. Therefore, the best possible overestimator over the line $y=\lambda x$ is given by the linear function interpolating the origin and the intersection of $(x,\lambda x)$ with either $x=U_x$ or $y=U_y$. We show that $\ave{f}_3(x,y)$ is a function satisfying this condition. In fact, if $U_y/U_x \geq \lambda$, then $y=\lambda x$ intersects first with $x=U_x$  and $\ave{f}_3(x,\lambda x) = \tfrac{\lambda x}{1+\lambda - \lambda U_x}$, which is a linear function, and $f_3(U_x,\lambda U_x) = \ave{f}_3(U_x,\lambda U_x)$. Otherwise, if $U_y/U_x \leq \lambda$, then $y=\lambda x$ intersects first with $y=U_y$  and $\ave{f}_3(x,\lambda x) = \tfrac{\lambda x}{1+\lambda - U_y}$, which is a linear function such that $f_3(U_y/\lambda, U_y) = \ave{f}_3(U_y/\lambda, U_y)$.
\qed\end{proof}

This idea of exploiting the convexity of $f_3$ over the rays $f_3(x,\lambda x)$ can also be extended to find concave envelopes for other functions satisfying this property; this is further elaborated in parallel work~\cite{barrera2021convex}. Note that, as anticipated, $\ave{f}_3(x,y)$ is an increasing function in both variables in $[0,U_x]\times[0,U_y]$; it can be easily verified that
\[\nabla \left(\frac{x\cdot y}{x+y-U_x\cdot y}\right) =\frac{1}{(-U_x y+x+y)^2} \left((1-U_x) y^2,x^2\right) \]
which is a nonnegative vector whenever $U_x\leq 1$. The other part of the definition of $\ave{f}_3(x,y)$ can be verified similarly.

In Figure \ref{fig:threeave}, we show the three concave envelopes we have discussed in this section.
Using these envelopes, we can formulate a mixed-integer convex nonlinear problem that can be solved more efficiently than the original model. This provides a tractable overestimation of the reliability of the resulting graph.

\begin{figure}[tbp]
\centering
	\includegraphics[height=3.5cm]{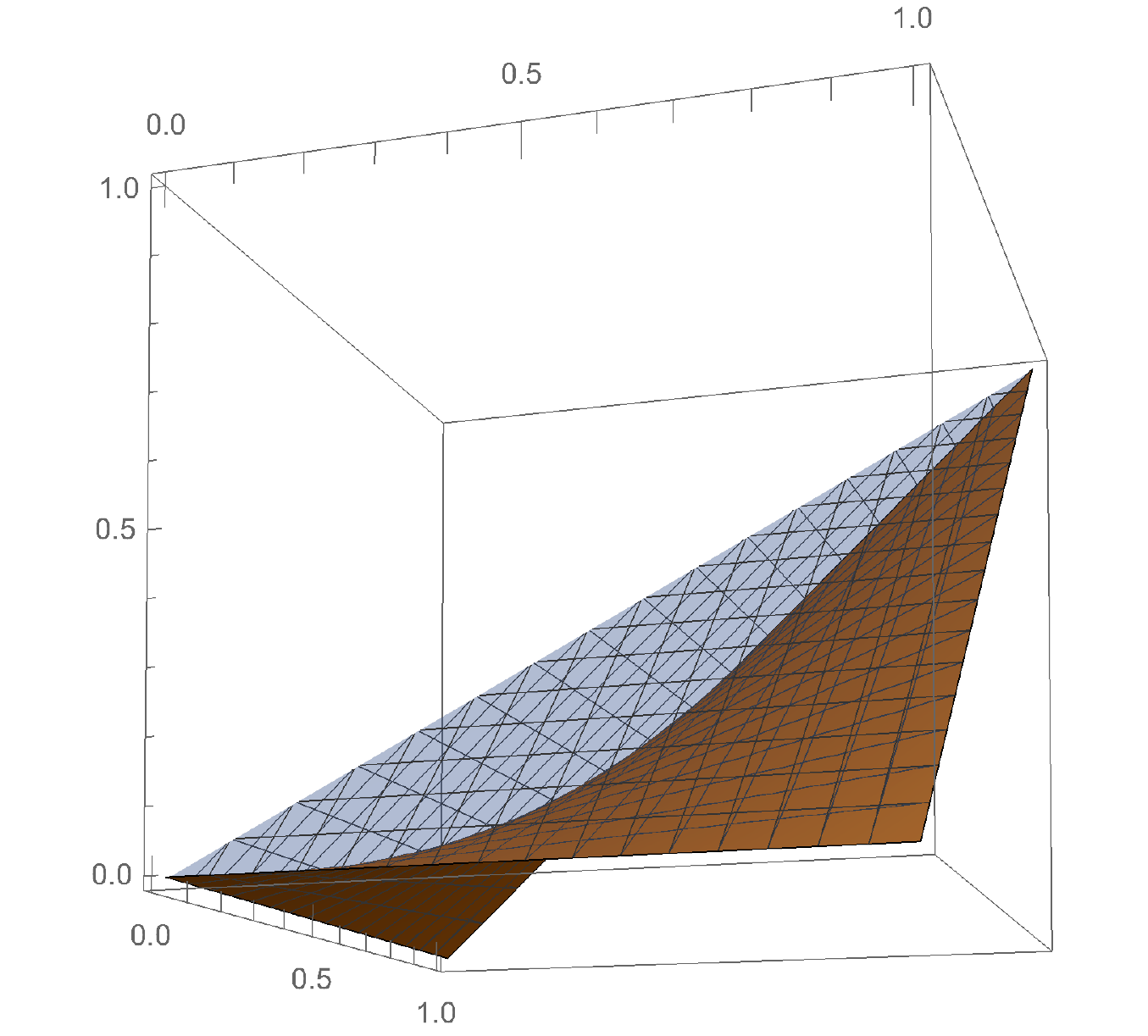}%
	\includegraphics[height=3.5cm]{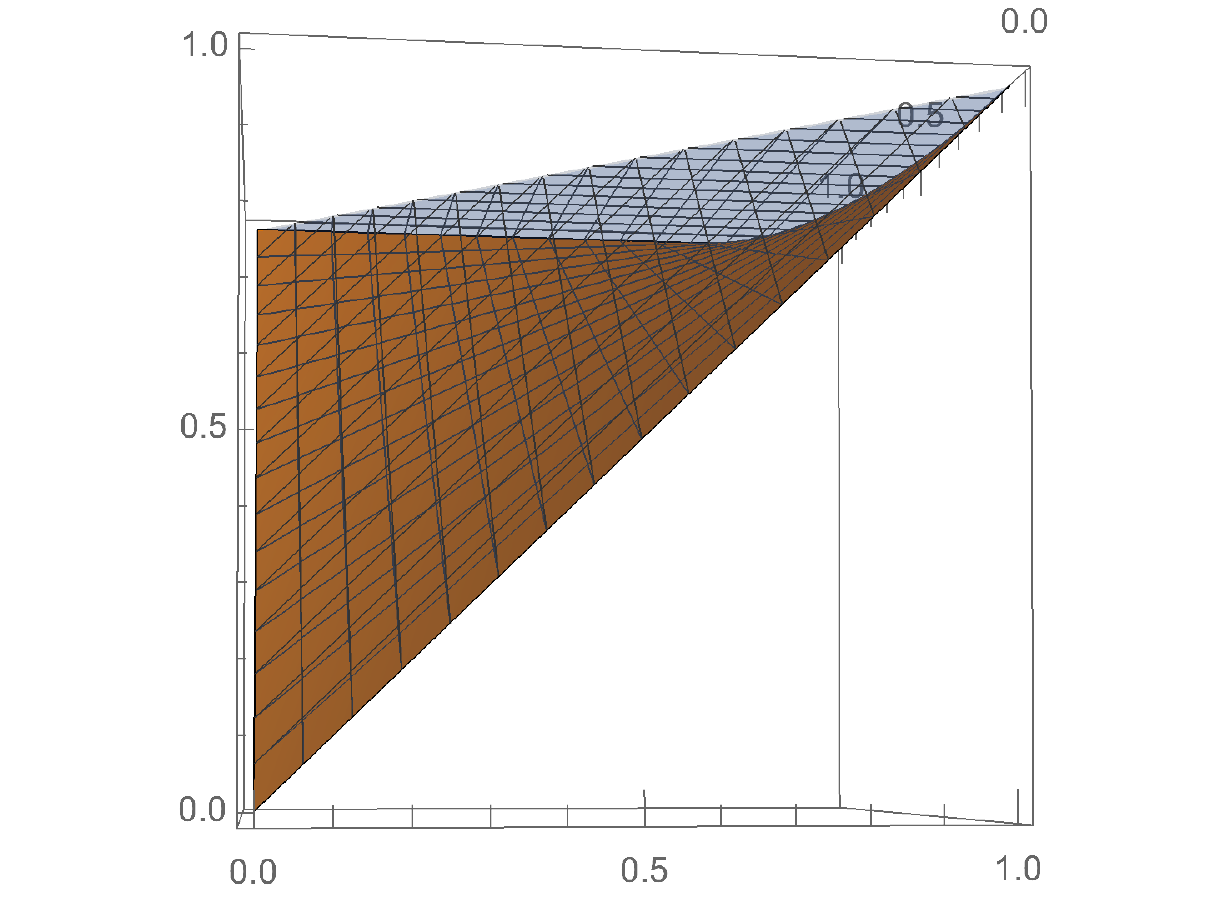}%
	\includegraphics[height=3.5cm]{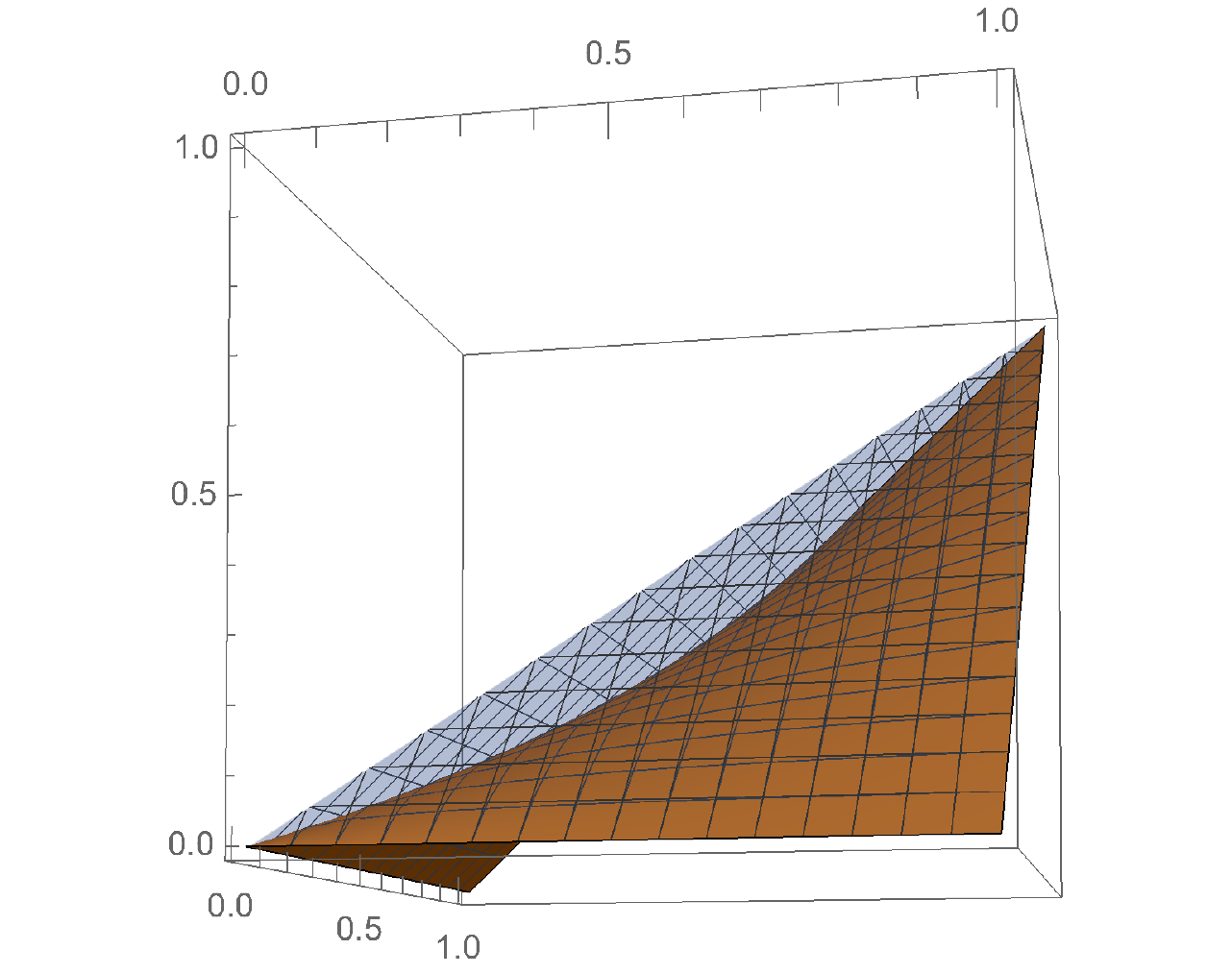}
	\caption{Concave envelopes for $x\cdot y$, $x+y-x\cdot y$ and $\tfrac{xy}{x+y-x\cdot y}$ for $0\leq x \leq 1$ and $0\leq y \leq 1$.} \label{fig:threeave}
\end{figure}

\subsection{A mixed integer convex approximation}\label{sec:mip}
Using the concave envelopes of the bivariate functions resulting from the series and parallel reductions, we can replace the corresponding constraint from model \eqref{eq:original_start}-\eqref{eq:original_finish} and obtain the following mixed-integer convex optimization approximation model:

\begin{subequations}\label{eq:convexrelax}
\begin{align}
	& \max\ R & \\
 & A X \leq b & \\
 Y_i &= p_i X_i, \quad \Omega_i=1,\quad \bar\Omega_i=1 &  &  &   i=1\ldots |E| \\
 Y_{|E|+i} &\leq Y_j (1-U_k)+ (1-U_j) Y_k + U_j U_k &  & &  i=1\ldots |E|-1: G_i = G_j +_P G_k \label{eq:final:parallel1}\\
 Y_{|E|+i} &\leq Y_j (1-L_k) +(1-L_j) Y_k  + L_j L_k&  & &   i=1\ldots |E|-1: G_i = G_j +_P G_k \label{eq:final:parallel2}\\
 \Omega_{|E|+i} &= 1,\quad \bar\Omega_{|E|+i} = \bar\Omega_{|E|+i-1}   & & &   i=1\ldots |E|-1: G_i = G_j +_P G_k \\
 Y_{|E|+i} &\leq \begin{cases} \frac{Y_j Y_k}{Y_j+Y_k - U_j Y_k} & \text{if } x/U_x \geq y/U_y  \\  \frac{Y_j Y_k}{Y_j+Y_k - Y_j U_k} & \text{if not} \end{cases} & &&  i=1\ldots |E|-1: G_i = G_j +_S G_k \label{eq:final:concave}\\
 \Omega_{|E|+i} &\leq Y_j (1-U_k)+ (1-U_j) Y_k + U_j U_k &&&  i=1\ldots |E|-1: G_i = G_j +_S G_k \\
 \Omega_{|E|+i} &\leq Y_j (1-L_k)+ (1-L_j) Y_k + L_j L_k &&&  i=1\ldots |E|-1: G_i = G_j +_S G_k \\
\bar\Omega_{|E|+i} &\leq U^{\bar\Omega}_{|E|+i-1}\cdot \Omega_{|E|+i} + \bar\Omega_{|E|+i-1}\cdot L^\Omega_{|E|+i} - U^{\bar\Omega}_{i-1}\cdot L^\Omega_i & && i=1\ldots |E|-1: G_i = G_j +_S G_k  \\
\bar\Omega_{|E|+i} &\leq L^{\bar\Omega}_{i-1}\cdot \Omega_i + \bar\Omega_{i-1}\cdot U^\Omega_i - L^{\bar\Omega}_{i-1}\cdot U^\Omega_i & && i=1\ldots |E|-1: G_i = G_j +_S G_k  \\
R &\leq U_{2|E|-1}\cdot \bar\Omega_{2|E|-1} + Y_z\cdot L^{\bar\Omega}_{2|E|-1} - U_{2|E|-1}\cdot L^{\bar\Omega}_{2|E|-1} & && \\
R &\leq L_{2|E|-1}\cdot \bar\Omega_{2|E|-1} + Y_z\cdot U^{\bar\Omega}_{2|E|-1} - L_{2|E|-1}\cdot U^{\bar\Omega}_{2|E|-1} & && \\
 Y_i&, \Omega_i, \bar{\Omega}_i\in[0,1]  & & &  i=1\ldots 2|E|-1\\ 
 X_i &\in \{0,1\}  &&&  i=1\ldots |E|
 \end{align}
\end{subequations}
where the constants $L_i$ and $U_i$ are valid lower and upper bounds for $Y_i$ and  $L^\Omega_i$, $U^\Omega_i$, $L^{\bar\Omega}_i$ and $U^{\bar\Omega}_i$ are valid lower and upper bounds for variables $\Omega_i$ and $\bar\Omega_i$, respectively. 
These upper bounds can be precomputed by assigning $U_i = p_i$ and $L_i=0$ for all $i\in 1\ldots |E|$ and then applying the corresponding functions $f_1$, $f_2$ or $f_3$ to these bounds for each series or parallel composition in $\mathcal{S}_G$.  

All constraints in previous model are linear, except for inequality~\eqref{eq:final:concave}. However, since $\ave{f}_3(x,y)$ is concave, we can enforce this nonlinear constraint with linear constraints given by its tangent hyperplane. Given a point $(x^*,y^*)$, we upper bound $\ave{f}_3(x,y)$ by the linear constraint $\ave{f}_3(x^*,y^*) + \frac{\partial\ave{f}_3}{\partial x}(x^*,y^*) \cdot (x-x^*) + \frac{\partial\ave{f}_3}{\partial y}(x^*,y^*) \cdot (y-y^*)$, which is 
\begin{align}
%	\bar{g}_3(x,y) &\leq \bar{g}_3(x^*,y^*) + \frac{\partial\bar{g}_3}{\partial x}(x^*,y^*) \cdot (x-x^*) + \frac{\partial\bar{g}_3}{\partial y}(x^*,y^*) \cdot (y-y^*)  \\
	\ave{f}_3(x,y)&\leq \begin{cases} \left(\frac{y^*}{x^*+y^*-U_x\cdot y^*}\right)^2 \cdot (1-U_x)\cdot x + \left(\frac{x^*}{x^*+y^*-U_xy^*}\right)^2 \cdot y & \text{if } x/U_x \geq y/U_y  \\
		\left(\frac{y^*}{x^*+y^*-x^*U_y}\right)^2 \cdot x + \left(\frac{x^*}{x^*+y^*-x^*U_y}\right)^2  \cdot (1-U_y) \cdot y & \text{if not} \end{cases} \label{eq:cuts}
\end{align}
These linear constraints can be added dynamically to optimization solvers during their optimization procedures.

\subsection{Computational experiments}\label{sec:results}

\subsubsection{Instances and implementation}
To test the effectiveness of our methodology, we generated random series-parallel graphs with $|E|=50$ and $|E|=100$ edges in the following way. The elementary reliabilities $p_e$ are generated uniformly at random between $0.9$ and $1.0$ for each edge $e$. To generate each instance, we start with a perfect bipartite matching.
%whose edges do not share a common node.
Clearly, the number of connected components in $G$ is initially $|E|$. Then, we iteratively select two components uniformly at random and connect them either in series or parallel (with equal probability), which diminishes the number of connected components in $G$ by one. 
The process is repeated until the resulting graph is connected. This procedure is detailed in Algorithm~\ref{alg:genGraphs}.

For the additional constraints (Eq.~\ref{eq:sideConstraints}), we impose a cardinality constraint that only a $\alpha$ fraction of the edges can be selected: $\sum_{e\in E} X_e \leq \alpha |E|$. Low values of $\alpha$ ($\alpha < 0.5$) lead to either infeasible problems or a reduced set of feasible solutions, which makes the problem easy to solve. Similarly, high values of $\alpha$ ($\alpha > 0.8$) encourage the solution to include most of the edges, also making the problem easy to solve. For these reasons, in our experiments, we use $\alpha=0.8$, which is a high value where our models behave very well, and $\alpha=0.6$, which yields the most challenging instances of our problem.

In our computational experiments, we compare two different configurations:
\begin{enumerate}
\item[Convex envelope cuts:] The model described in Section~\ref{sec:mip}, including the cuts from \eqref{eq:cuts} to approximate our concave envelope from Theorem~\ref{thm:main}.
%\item[Improved envelope cuts:] The previous model, including the improvements on the bounds provided by the branch-and-bound tree from Subsection~\ref{sec:dynamicCuts}.
\item[Without cuts:] The model from Section~\ref{sec:mip}, considering a simpler constraint to bound the nonlinear concave constraint~\eqref{eq:final:concave}, which we describe next.
\end{enumerate}
The second configuration is constructed to understand the effectiveness of the concave approximation provided in Theorem~\ref{thm:main} for series compositions. This simpler general approximation for $f_3(x,y)$ is given by constraints $f_3(x,y) \leq x$ and $f_3(x,y) \leq y$; these constraints are obtained when replacing either $x$ or $y$ in the denominator of $f_3(x,y)$ with 1 or by replacing $U_x$ or $U_y$ with 1 in $\ave{f}_3(x,y)$. This corresponds to the tangent hyperplanes of $\ave{f}_3(x,y)$ on $(x,y) = (0,U_y)$ and $(x,y)=(U_x,0)$.

\begin{algorithm}
\caption{Generation of random instances of size $k$}\label{alg:genGraphs}
\begin{algorithmic}
%\ForAll{$e\in E$}\State $p_e \gets U[0.9,1]$\EndFor
\State $E \gets \{e_1,\ldots e_k\}$ (disjoint edges) %single edges, not connected among them
\State $p_e \gets U[0.9,1]$ for all $e\in E$ \Comment{Random initial reliabilities}
\State $\mathcal{G} \gets E$
\While{$|ConnectedComponents(\mathcal{G})|> 1$}
\State $g_1,g_2 \gets$ random subgraphs from $\mathcal{G}$ (chosen equiprobable)
\State $u \gets U[0,1]$
\If {$u<0.5$} 
\State $g = g_1 +_P g_2$ \Comment{Parallel connection}
\Else
    \State $g = g_1 +_S g_2$ \Comment{Series connection}
\EndIf
\State $\mathcal{G} \gets \mathcal{G} \setminus \{g_1 \cup g_2\} \cup g$ \Comment{Replace $g_1$ and $g_2$ in $\mathcal{G}$ by $g$}
\EndWhile
\end{algorithmic}
\end{algorithm}

These models were implemented with Python 3.7 using the IBM\textsuperscript{\tiny\textregistered} Decision Optimization CPLEX\textsuperscript{\tiny\textregistered} Modeling for Python (DOcplex.MP) v2.11 of CPLEX Studio v12.10~\cite{docplex}.
All CPLEX parameters have their default values, and no cut manager was implemented for the additional cuts~\eqref{eq:cuts}. 
All computations were made on machines running Linux under x86\_64 architecture running in a single thread.

\subsubsection{Computational results} \label{sec:firstexperiments}
%Since the constraints of the approximated problem provide upper bounds on the value of the variables, we begin by validating the quality of the solution provided as an approximation of the true reliability of the resulting graph. In order to do that, in this first set of experiments, we omit the Combinatorial Benders cut presented in Subsection~\ref{sec:trueReliability} and we compare the results of the three configurations.

\begin{figure}[htbp]
	\centering
	\includegraphics[height=4.5cm]{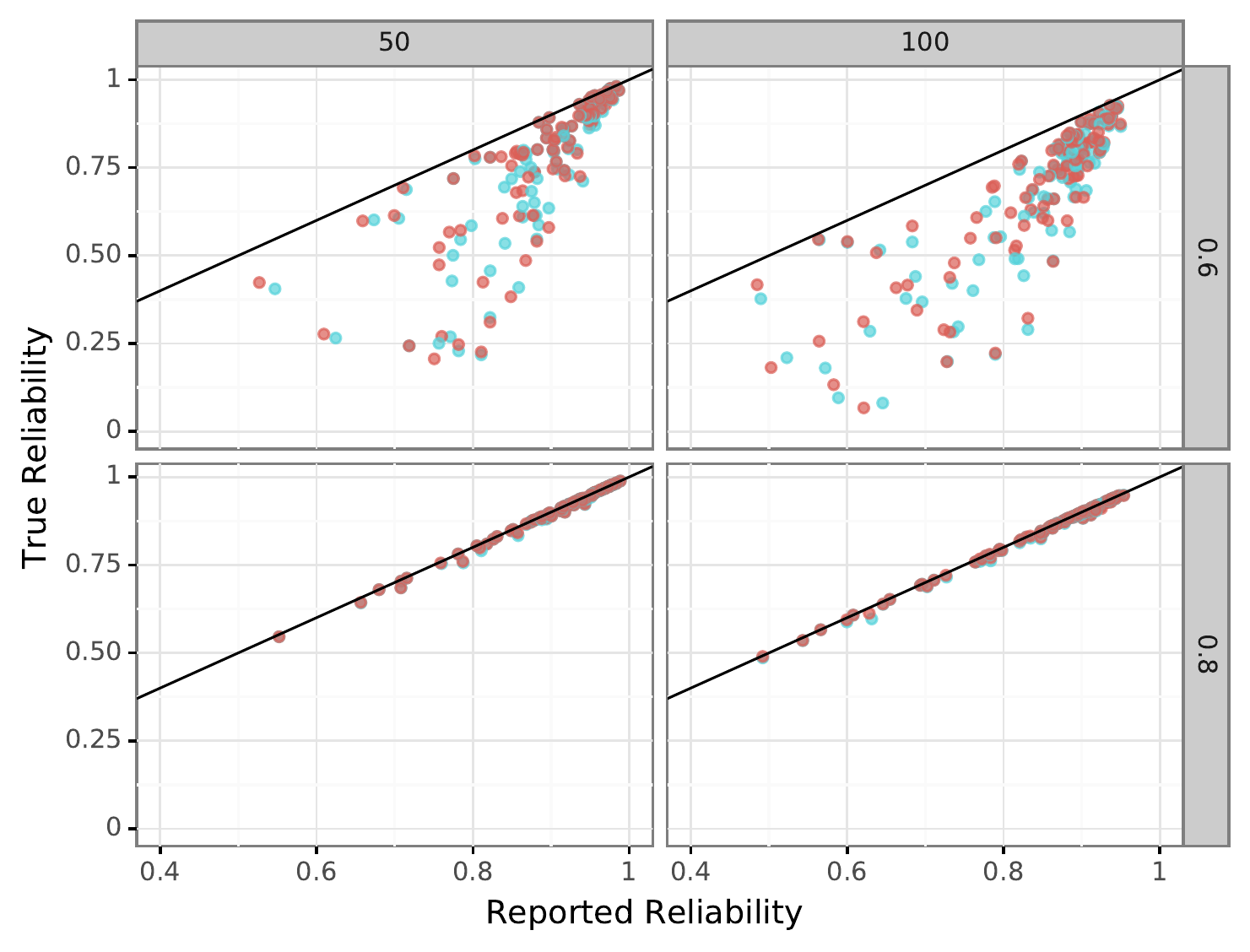}~%
	\includegraphics[height=4.5cm]{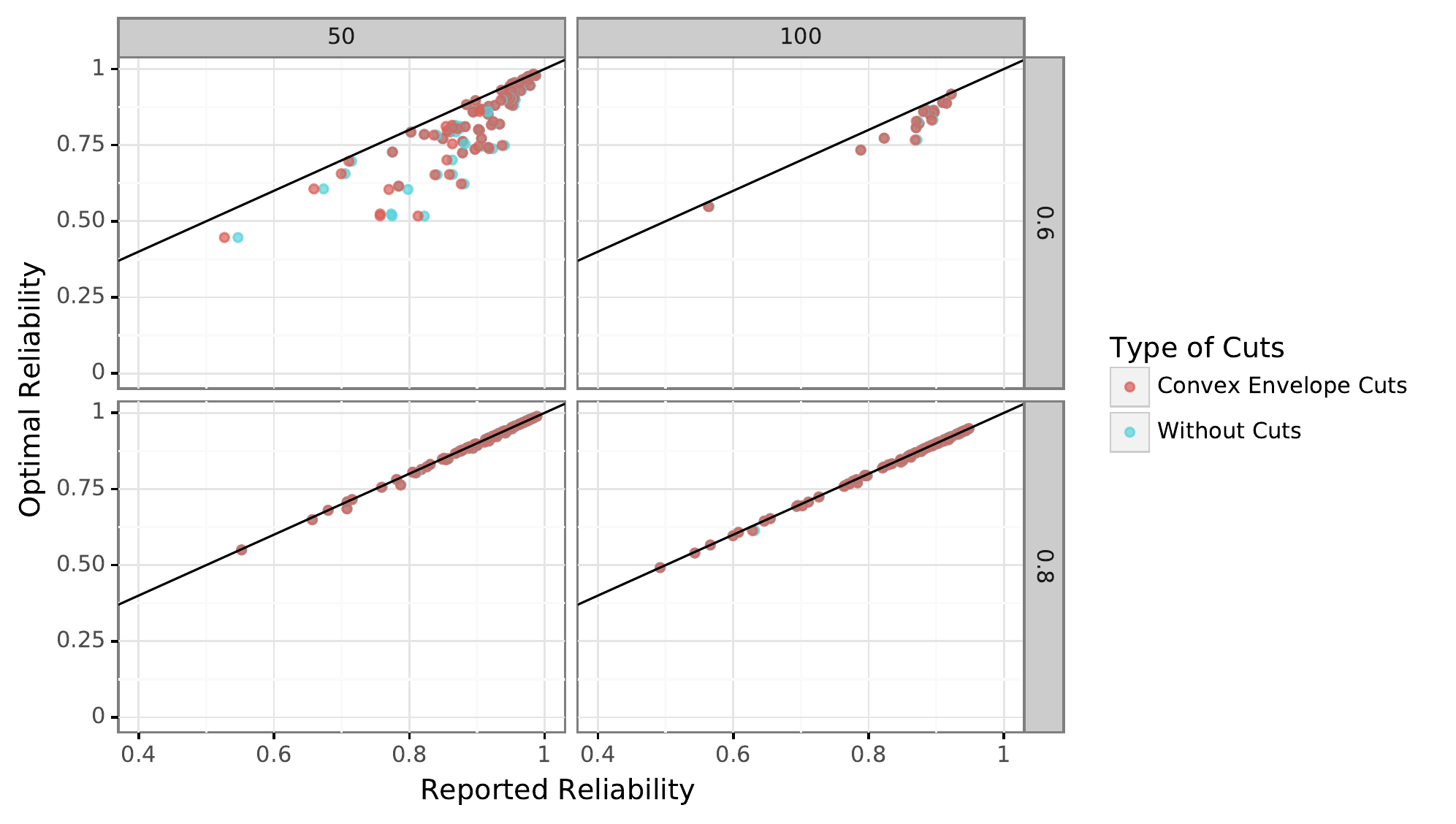}
	\caption{Comparison between the reported objective value and the true reliability and the optimal reliability of the resulting network}\label{fig:results}
\end{figure}

Figure~\ref{fig:results} (left) compares the objective value reported by our model (x-axis) for each instance versus the true reliability of the resulting solution for the two configurations described. Similarly, Figure~\ref{fig:results} (right) compares the objective value reported versus the true optimal solution of the problem, for the cases where the latter is known.\footnote{True optimal solutions are obtained using the results of the following sections but presented here for illustrative purposes.} Each subfigure is divided in four, depending on the number of edges (horizontal) and the budget (vertical) of the problem.

In these experiments, we can see that both configurations behave extremely well for a budget of $\alpha=0.8$: the resulting objective value of the model is not only very close to the real reliability of the solution but also very close to the true optimal solution of the problem. These relative errors are less than $1\%$ for all instances and configurations evaluated. 

For $\alpha=0.6$, this behavior changes drastically. The reported objective value of the problem (an upper bound on the reliability) differs considerably for many instances in both configurations. Moreover, it seems that this effect is more pronounced for lower reliabilities. A similar effect is observed when we compare the reported reliability with the optimal reliability of the problem.\\ 

These experiments show that, in challenging instances, the overestimation of the reliability provided by the concave envelopes alone is not good enough for estimating the true reliability of the resulting graphs, leading to suboptimal solutions for our problem. Moreover, they indicate that little is gained from using cutting planes to iteratively approximate $\ave{f}_3$, in comparison to its simple approximation based on two tangent hyperplanes, at least in this setting. In the next section, we study different improvements to avoid these issues and where the cutting planes associated with $\ave{f}_3$ have considerably more impact. These improvements lead the model to the true optimal solution in most instances.

\subsection{Further improvements to the model}\label{improvements}

\subsubsection{Providing the true reliability of the solution}\label{sec:trueReliability}
The previous model considers overestimators of the nonlinear functions, which can lead to overestimations of the true reliability of the resulting network. However, based on the selected edges of the graph (the values of the $X$ variables), we can compute the ``true'' reliability of the resulting graph, and we can improve our model by introducing \emph{combinatorial Benders cuts}~\cite{codato2006combinatorial} to avoid this problem.

Specifically, given a feasible solution $\hat{X}$ we can compute the resulting reliability $\mathcal{R}_{\hat{X}}=\mathcal{R}(p_1\hat{X}_1,\ldots,p_{|E|}\hat{X}_{|E|})$ in linear time using Algorithm~\ref{alg:computeReliability} and obtain the following valid inequality---a combinatorial Benders cut associated with $\hat{X}$---that bounds the value of $R$
\begin{equation}
    R \leq \mathcal{R}_{\hat{X}} + \sum_{e\in E : \hat{X}_e=1} (1-X_e) + \sum_{e\in E : \hat{X}_e=0} X_e 
\end{equation}
This constraint implies that if $\hat{X}$ is the optimal solution to the problem, then $R=\mathcal{R}_{\hat{X}}$. For other feasible solutions $X^*\neq \hat{X}$, this constraint is trivially satisfied because $R \leq 1$. 
%
% We add these cuts dynamically whenever a feasible integer solution has been found in the model as \emph{lazy cuts}, which can be implemented in most of the current linear optimization solvers.
%
We can strengthen this inequality using the following observation. Note that the all-terminal reliability of a graph $\mathcal{R}(p_1X_1,p_2X_2,\ldots ,p_{|E|} X_{|E|})$ only increases when a new edge is added. 
%Hence, this upper bound on the reliability comes from edges not included in the current solution $\hat{X}$. In other words, this bound is valid for all solution $X$ not including the edges omitted in $\hat{X}$. 
In other words, the reliability of any solution $X$ such that $\hat{X}_e = 0 \Rightarrow X_e = 0$ must be smaller than that of $\hat{X}$. Therefore, the combinatorial cuts can be strengthened to
\begin{equation}
    R \leq \mathcal{R}_{\hat{X}} + \sum_{e\in E : \hat{X}_e=0} X_e. \label{eq:CombinatorialCut}
\end{equation}

By the same reasoning, if a feasible solution of the problem contains all edges selected in $\hat{X}$, then its reliability will be at least $\mathcal{R}_{\hat{X}}$. Hence, we can enforce this lower bound for the reliability by adding the following constraint each time that an incumbent solution $\hat{X}$ has been found:
\begin{equation}
    R \geq \mathcal{R}_{\hat{X}} - \sum_{e\in E : \hat{X}_e=1} (1-X_e). \label{eq:CombinatorialCutLB}
\end{equation}

% Note that constraints \eqref{eq:CombinatorialCut} and \eqref{eq:CombinatorialCutLB} are valid because the reliability of the graph $\mathcal{R}$ as a function of $X$ is monotone, in the sense that if $X^1 \leq X^2$ then $\mathcal{R}(X^1) \leq \mathcal{R}(X^2)$. 

This reasoning can also be expanded to other variables. The monotonicity exhibited by the reliability function also holds for $Y_i$, $\Omega_i$ and $\bar{\Omega}_i$ as functions of $X$, because all three functions $f_1(x,y)$, $f_2(x,y)$ and $f_3(x,y)$ are nondecreasing in both dimensions. Therefore, given a solution $\hat{X}$, we can apply Algorithm~\ref{alg:computeReliability}  to compute the values  of variables $Y$, $\Omega$ and $\bar{\Omega}$ associated with this solution and derive similar cuts for all variables of the problem each time that a new incumbent solution $\hat{X}$ is found during the branch-and-bound process.

%\todo[inline]{Podría ser mejor $ R \leq \mathcal{R}_{\hat{x}} + \sum_{e\in E : \hat{x}_e=0} p_e x_e$ ?  Dado que la reliability no puede aumentar en mas de $p_e$ si se incluye un nuevo arco? O incluso calcular $\mathcal{R}_{\hat{x}+x_e}$ para cada $e$ y poner ese factor? ¿Es una cota válida? Creo que no. ¿Es submodular?.  Respuesta: No es una cota valida. }

\begin{remark}
Including these inequalities during a branch-and-bound procedure ensures that we obtain an optimal solution to the original problem \eqref{eq:mainmodel} at the end of the optimization routine, unless a time limit is reached. Whenever a feasible solution $\hat{X}$ to the relaxation is found, the combinatorial Benders cut ensures that $R=\mathcal{R}_{\hat{X}}$ if $\hat{X}$ is reported optimal. The resulting optimization routine might be impractical though, as it might resort to a costly enumeration if the relaxation \eqref{eq:convexrelax} is not tight enough. The next improvements aim at better approximating the problem and making the tree-search more efficient.
\end{remark}

\subsubsection{Improving inequalities on the branch-and-bound tree}\label{sec:dynamicCuts}
Our proposed model considers the best possible concave overestimators for each function. These envelopes depend on the lower and upper bounds for each variable, and even if these bounds are tight, the envelopes may not provide a tight approximation of the functions in the whole feasible region.
% For example, depending on the additional constraints of the problem, in most cases the best lower bound for $Y_i$ is $L_i=0$ for $i=1\ldots |E|$.
Nevertheless, the branch-and-bound procedure of MILP solvers is based on imposing new bounds on the variables while branching and thus improving the relaxations in each node. These bounds are only valid locally, but we can use them to obtain the concave envelopes based on these new bounds, which yields better local approximations of each function. 

If the branch-and-bound process fixes a variable $X_e$ to 0 or 1, we can propagate this decision to improve the lower and upper bounds for all variables corresponding to reductions that include this edge. 
%We recompute the concave envelopes using these new bounds as mentioned above. 
% Note that these new constraints are only valid for the subtree of the branch-and-bound fixing this variable, so they should be added as \emph{local cuts} of the problem.
%
%In particular, if the improved bounds fix the value of a variable (
For instance, if the fixed variable is $X_j=1$, then $Y_j=p_j$. Thus, for the case of a parallel reduction, we can improve \eqref{eq:final:parallel1} and \eqref{eq:final:parallel2} by adding the local linear constraint $Y_i = 1-(1-p_j)(1-Y_k)$. In the case of a series reduction, we can improve \eqref{eq:final:concave} and its associated cuts \eqref{eq:cuts} by adding the inequality 
\begin{equation}
 Y_i \leq \left(\frac{p_j}{Y_k^*+p_j-Y^*_k\cdot p_j}\right)^2 \cdot (Y_k-Y_k^*) + \frac{p_j Y_k^*}{Y_k^*+p_j-Y^*_k\cdot p_j} \label{eq:convcutImproved}
\end{equation}
where $Y_k^*$ is the current solution of variable $Y_k$ at the node of the branch-and-bound tree. This is possible because the function $f_3(p,y)$ is concave on $y$ for any fixed value of $p$, so \eqref{eq:convcutImproved} corresponds to the gradient of this function on $y=Y_k^*$. Similar improvements can be included for the remaining equations involving $\Omega_i$ and $\bar{\Omega}_i$.

Note that our concave envelopes for series reduction (Theorem~\ref{thm:main}) only apply when the lower bounds are equal to $0$, so this approximation cannot be improved if the lower bounds are improved. However, we can still add linear constraints in this case: since $Y_i=f_3(Y_j,Y_k)$ is increasing and concave for a fixed $Y_j$ or $Y_k$, we can overestimate this function by two hyperplanes tangent to the point $(L_j,L_k)$.
\begin{align*}
    Y_i &\leq \left.\frac{\partial f_3}{\partial x}\right|_{(L_j,L_k)} \cdot (Y_j-L_j) + \left.\frac{\partial f_3}{\partial y}\right|_{(U_j,L_k)} \cdot (Y_k-L_k) + f_3(L_j,L_k) \\ 
    &= \frac{L_k^2}{(L_j + L_k - L_j L_k)^2}\cdot (Y_j-L_j) + \frac{U_j^2}{(U_j + L_k - U_j L_k)^2} \cdot(Y_k-L_k) + \frac{L_j L_k}{L_j+L_k-L_jL_k}
\end{align*}
To see that these cuts are valid, note that function $f_3(x,y)$ is concave on $x$ for a fixed $y$, so in particular, this is a valid upper bound for $y=L_k$. On the other hand, the partial derivative $\frac{\partial f_3}{\partial y}$ is increasing with respect to $x$, attaining its maximum value on $x=U_j$. Since $f_3$ is also concave for a fixed $x$, then it is a valid bound for $x=U_j$ and then for all $(x,y)\in [L_j,U_j]\times [L_k,U_k]$.

Interchanging the roles of $x$ and $y$ in $f_3(x,y)$, we can also bound $Y_i$ for a series composition by
\begin{align*}
    Y_i &\leq \left.\frac{\partial f_3}{\partial x}\right|_{(L_j,U_k)} \cdot (Y_j-L_j) + \left.\frac{\partial f_3}{\partial y}\right|_{(L_j,L_k)} \cdot (Y_k-L_k) + f_3(L_j,L_k) \\ 
    &= \frac{U_k^2}{(L_j + U_k - L_j U_k)^2}\cdot (Y_j-L_j) + \frac{L_j^2}{(L_j + L_k - L_j L_k)^2}\cdot (Y_k-L_k) + \frac{L_j L_k}{L_j+L_k-L_jL_k}
\end{align*}
See Figure~\ref{fig:envelopeImproved} for an example on how these hyperplanes improve the overestimation of $f_3$ on $(L_j,L_k)$.
% \todo[inline]{GM:Esta figura no se ve muy bien. Quizas cambiar el ángulo y/o los colores/transparencia. EM: Agregue otra. ¿se ve mejor ahí? Si}

\begin{figure}[tbp]
    \centering
	\includegraphics[height=3.5cm]{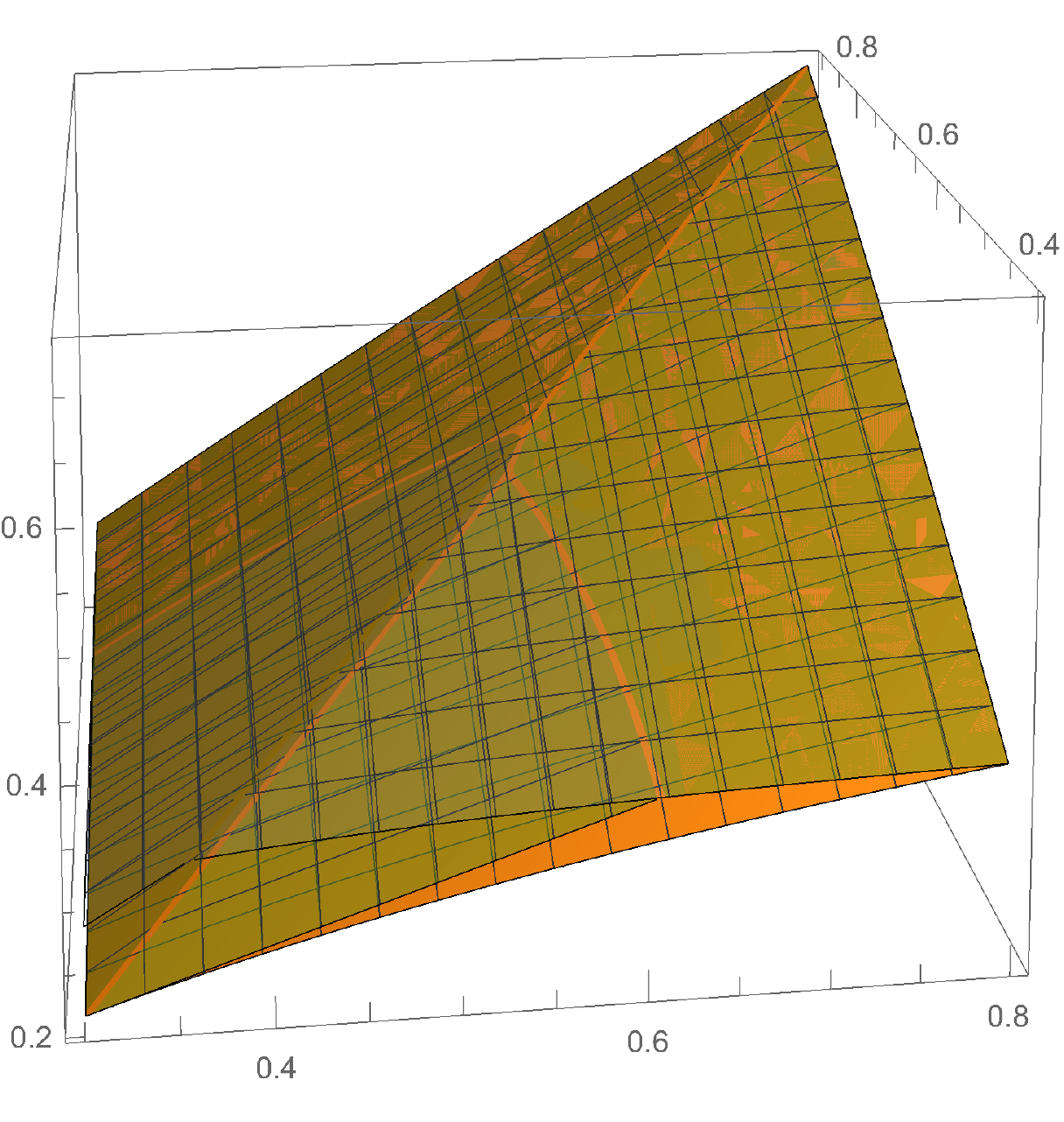} \hspace{1cm}%
	\includegraphics[height=3.5cm]{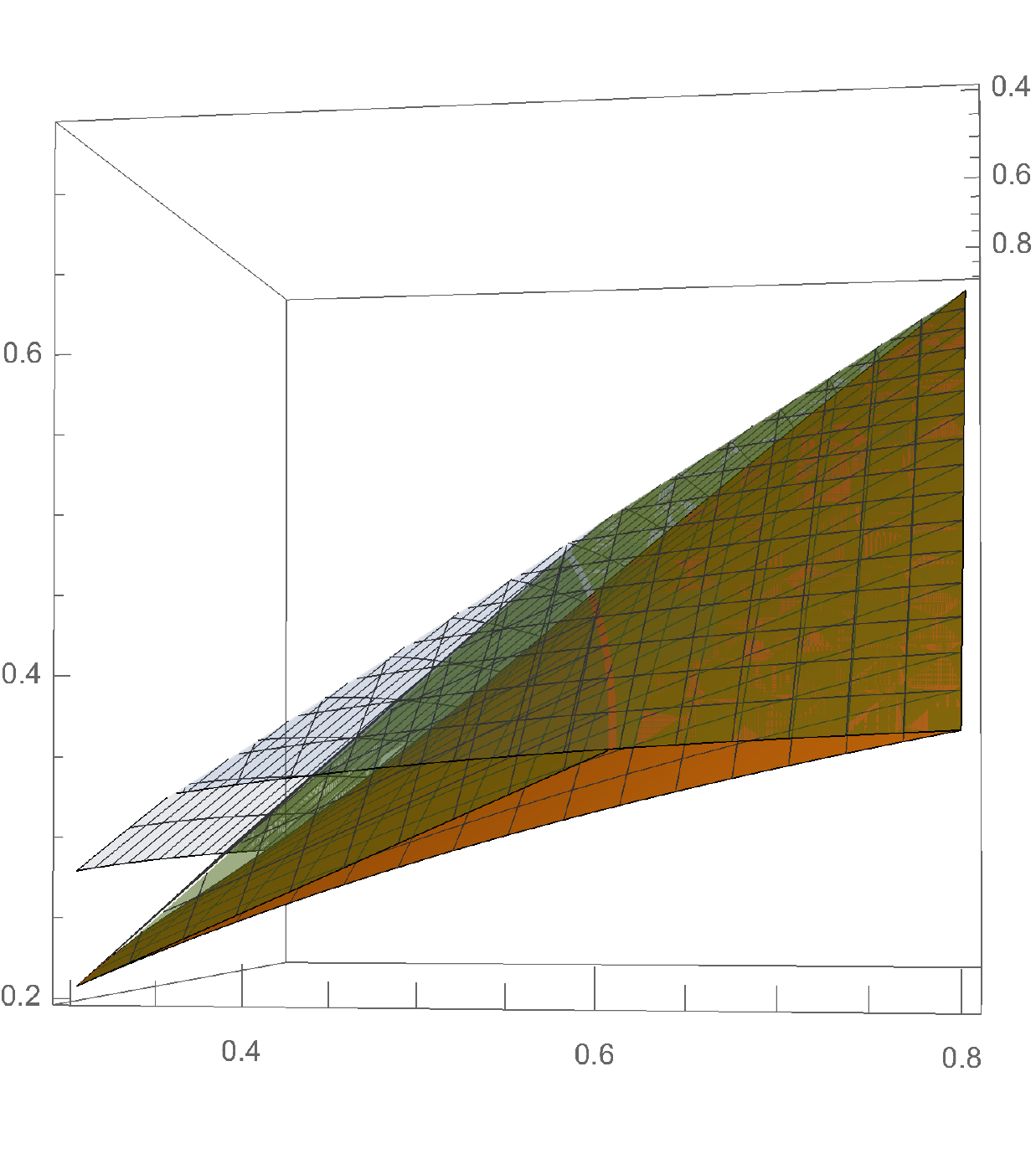}
	\caption{Function $\tfrac{xy}{x+y-x\cdot y}$ for $0.3\leq x \leq 0.8$ and $0.4\leq y \leq 0.9$ (in orange), its concave envelope over $[0,0.8]\times [0,0.9]$ (in blue) and the additional tangent hyperplanes at $(0.3,0,4)$ (in green)}\label{fig:envelopeImproved}
\end{figure}

\subsection{Computational experiments for the true optimal solution for the problem}

In this second set of experiments, we now include the combinatorial Benders cut \eqref{eq:CombinatorialCut} to ensure that the optimal solution provides the true reliability for the problem. We include these cuts for the previous two configurations \textbf{convex envelope cuts} and \textbf{without cuts}, and we add a third configuration:
\begin{description}
\item[Improved envelope cuts:] The model described in Section~\ref{sec:mip}, including the improvements based on the local bounds provided by the branch-and-bound tree (see Subsection~\ref{sec:dynamicCuts}).
\end{description}

Combinatorial cuts for computing the exact reliability (\S\ref{sec:trueReliability}) are implemented as \texttt{LazyContraintCallback}, and improved cuts in branch-and-bound (\S\ref{sec:dynamicCuts}) are implemented as \texttt{UserCutCallback}. For this set of experiments, we set a time limit of 3 hours for each problem.

Additionally, to benchmark our models with other solvers, we solve the original model \eqref{eq:original_start}-\eqref{eq:original_finish} using the MINLP solver SCIP v7.02~\cite{scip} compiled with the parameters for better performance for this kind of nonlinear nonconvex problem.

\begin{figure}[htbp]
	\centering
	\includegraphics[height=5.5cm]{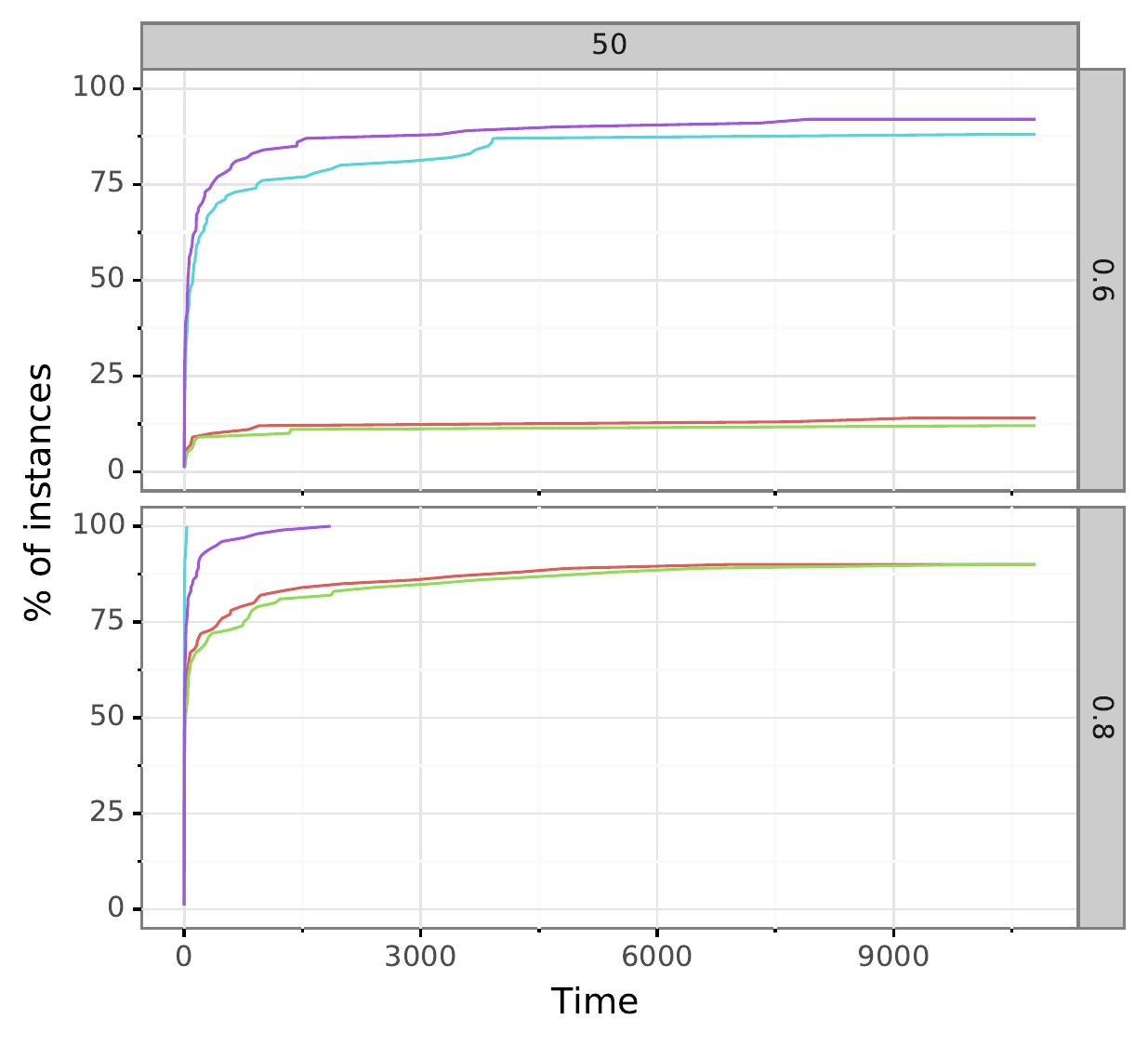}~%
	\includegraphics[height=5.5cm]{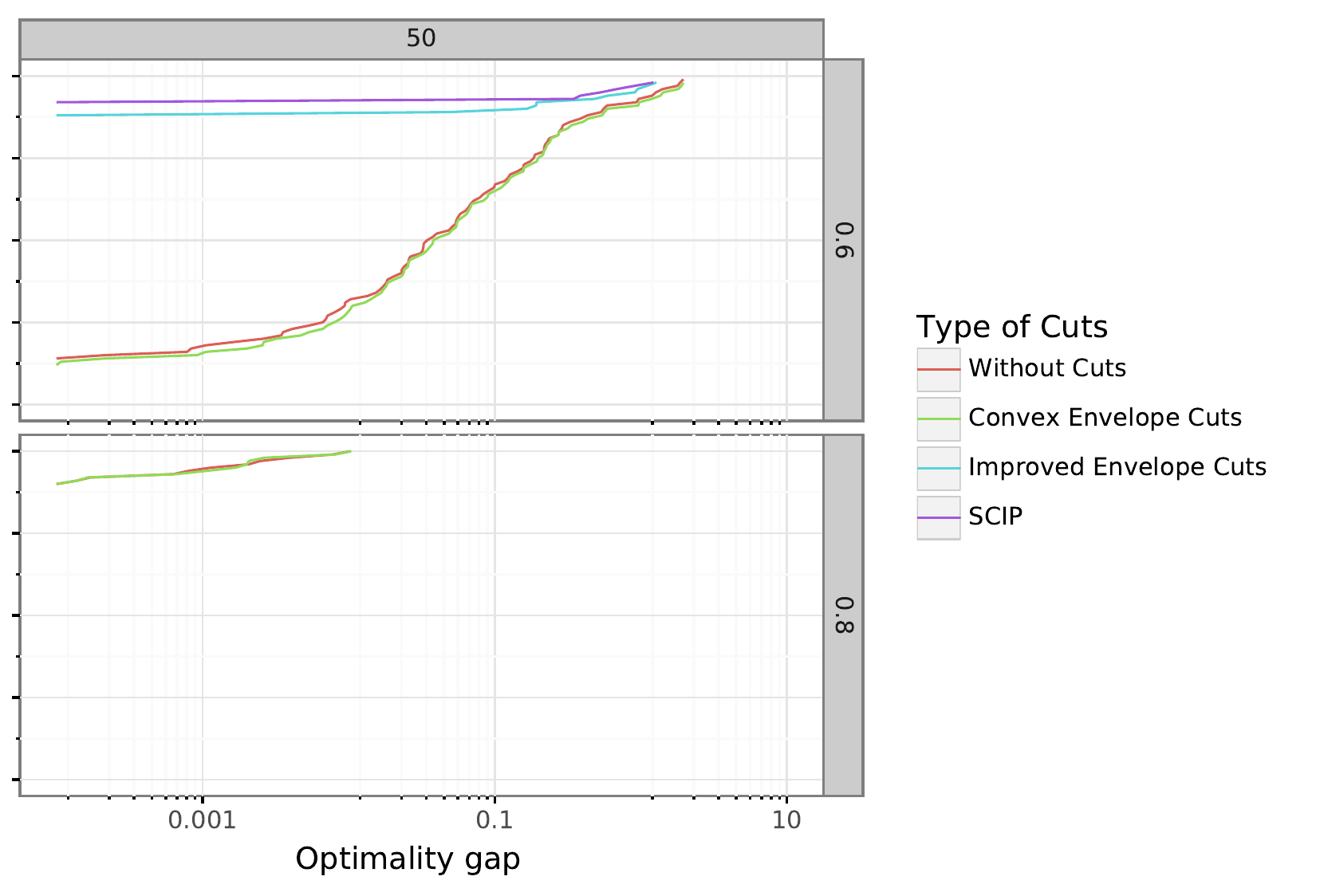}
	\caption{Performance profiles of the different configurations ($|E|=50$)}\label{fig:resultsReal50}
\end{figure}

Figure~\ref{fig:resultsReal50}  shows the performance profiles of the different configurations for instances with $|E|=50$ edges. The figure on the left shows the percentage of instances solved up to optimality before a given time ($x$-axis). For the instances that are not solved within the time limit, the figure on the right shows the percentage of instances attaining a given optimality gap (in log scale). 

Let us first analyze the instances for $\alpha=0.8$.  As shown in previous experiments, for these cases, all methods provide a good approximation of the true and optimal reliabilities. Therefore, the configurations without the improved envelope cuts behave well, solving 90\% the instances within the time limit of 3 hours, and the unsolved instances present a small optimality gap---a 2\% gap in the worst case. However, note that there is a nonnegligible portion of instances that require more than an hour to solve; this indicates that the model is able to find a good solution but that it cannot quickly prove its optimality because it needs to visit a large branch-and-bound tree to discard all other potential solutions. On the other hand, the improved envelope cuts behave drastically differently, solving all instances in less than one minute. This can be explained by this configuration's ability to locally adapt concave envelopes during the branch-and-bound tree, providing better estimations and thus better bounds, which yield a smaller branch-and-bound tree. 

This better approximation of the nonconcave functions is even more relevant for $\alpha=0.6$. While the combinatorial Benders cuts help in fixing the mismatch between the reported reliability and the true reliability discussed in Section \ref{sec:firstexperiments}, the weak estimations provided by the global approximation of the functions prevent the solver from finding better solutions and/or proving optimality efficiently: only 14\% of the instances are solved within the time limit when the improved envelope cuts are not included. This number increases to  88\% when these cuts are included, with most of these instances being solved in just a few minutes.

%\todo[inline]{GM: Hay algo que me molesta de este último analisis. El problema de reported reliability vs true reliability lo debería arreglar los Combinatorial Cuts (que están en todas), no los Improved Envelope Cuts. Eso al menos es el mensaje que dimos. Creo que quedaría mejor algo como:\\
%\vskip .2cm
%This better approximation of the non-concave functions is even more relevant for $\alpha=0.6$. While the combinatorial Benders cuts help in fixing the mismatch between the reported reliability and the true reliability discussed in Section \ref{sec:firstexperiments}, the weak estimations provided by the global approximation of the functions prevent the solver from finding better solutions and/or proving optimality efficiently: only 14\% of the instances are solved within the time limit when the Improved Envelope Cuts are not included. This number increases to  88\% when these cuts are included, with most of these instances being solved in just a few minutes.
%}

\begin{figure}[htbp]
	\centering
	\includegraphics[height=5.5cm]{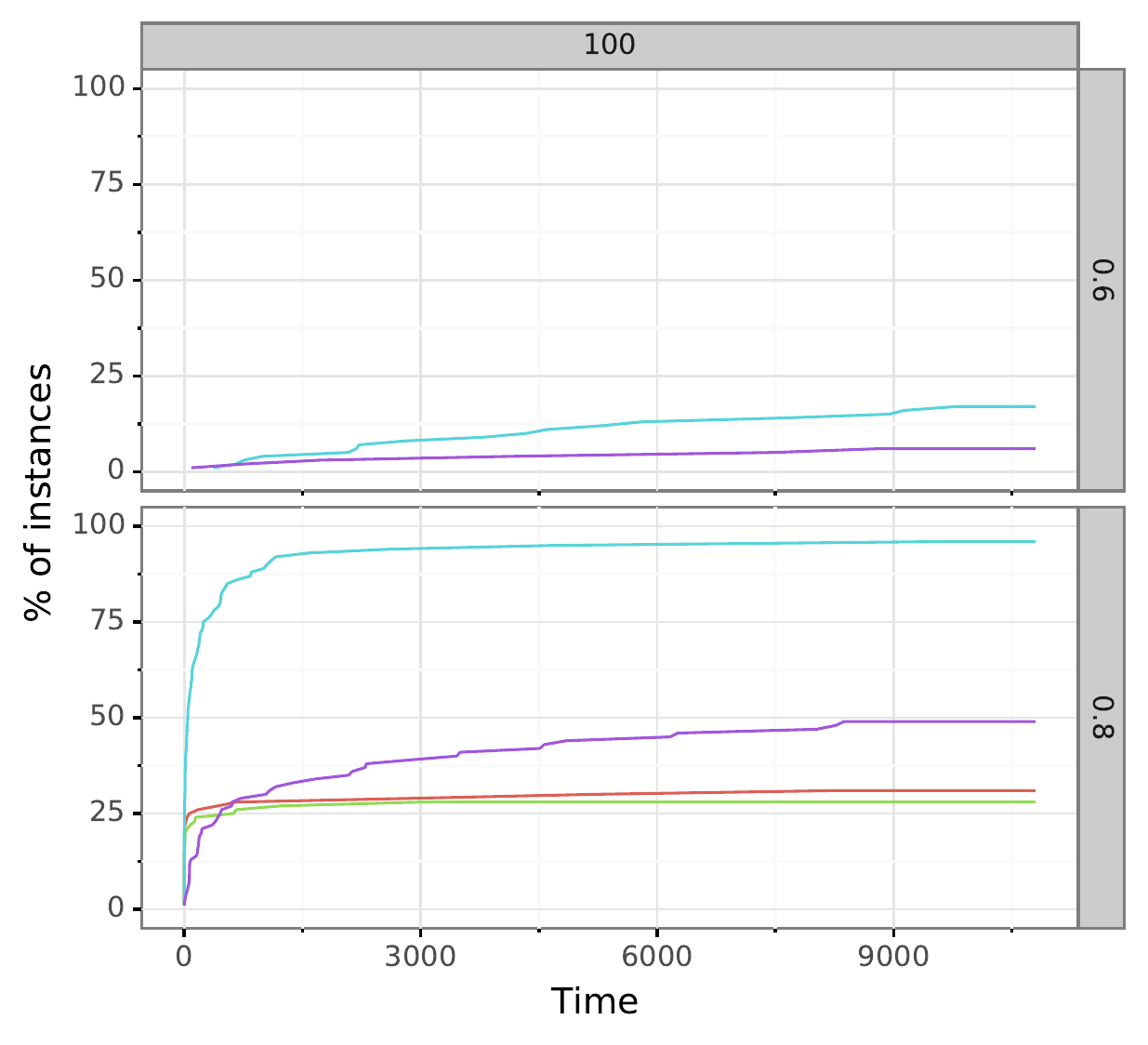}~%
	\includegraphics[height=5.5cm]{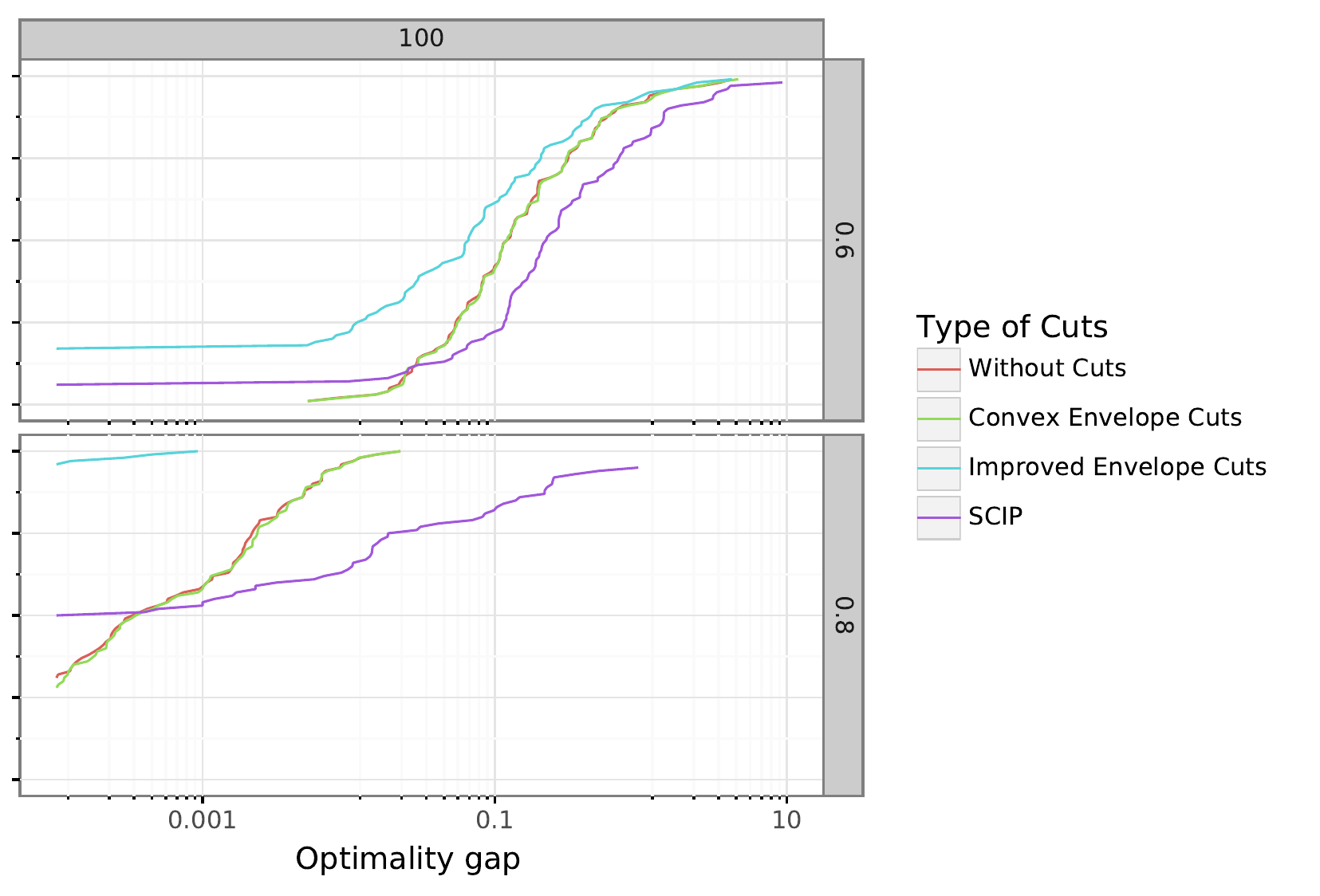}
	\caption{Performance profiles of the different configurations ($|E|=100$)}\label{fig:resultsReal100}
\end{figure}

The dominance of the improved envelope cuts also occurs for the instances with $|E|=100$ edges. We present these results in Figure~\ref{fig:resultsReal100}. We first note that the problem is considerably harder to solve. For instance, the configuration \emph{without cuts} can solve only 31\% of the instances for $\alpha=0.8$ within the time limit of three hours. Adding the convex envelopes of Theorem~\ref{thm:main} improves this metric, but only marginally. Nevertheless, the optimality gap obtained by these configurations is good, with more than 95\% of the instances finishing with a gap of less than 1\%. In the improved envelope cuts setting, 96\% of the instances are solved for $\alpha=0.8$, and the remaining instances finish with an optimality gap less than $0.1\%$.  

The problem becomes much more challenging for $\alpha=0.6$ and 100 edges. This is expected, because due to the sequential construction process of the graph, the differences between the nonlinear functions and their concave envelopes are propagated into the overall approximation quality and become more relevant when the number of steps in the construction sequence (i.e., the number of edges) is large. In fact, without considering the improved envelope cuts, the solver is not able to solve any instance, and the optimality gaps are substantial for most of the cases. The performance improves when including the improved envelope cuts, resulting in 17\% of the instances being solved and obtaining better optimality gaps for the unsolved instances. 
%This experiment allows to show the limits of the proposed approach for this problem.  

\subsubsection{Comparison with MINLP solver}

To benchmark the proposed model against current state-of-art solvers for nonlinear optimization models, we also solve the problem using the SCIP solver. SCIP is among the best general-purpose solvers that are able to deal with nonconvex constraints. It implements multiple bounding techniques, some of which are similar to those studied in this paper, along with spatial branch-and-bound based on linear outer-approximations of the problem. For more details, see~\cite{vigerske2018scip}. 

Figures~\ref{fig:resultsReal50}~and~\ref{fig:resultsReal100} show performance profiles of SCIP in comparison with our approach. We can see that for $|E|=50$, SCIP behaves in a similar way to the improved envelope cuts, being slightly slower for $\alpha=0.8$. However, for larger problems with $|E|=100$ edges, SCIP's performance decreases considerably, and it is outperformed by our proposed improved envelope cuts. SCIP can solve only half of the instances for $\alpha=0.8$ and only 6\% of the instances for $\alpha=0.6$, reaching the time limit with optimality gaps that are worse than the basic configuration without cuts, in most cases. 

\begin{figure}[htbp]
	\centering
	\includegraphics[width=.48\linewidth]{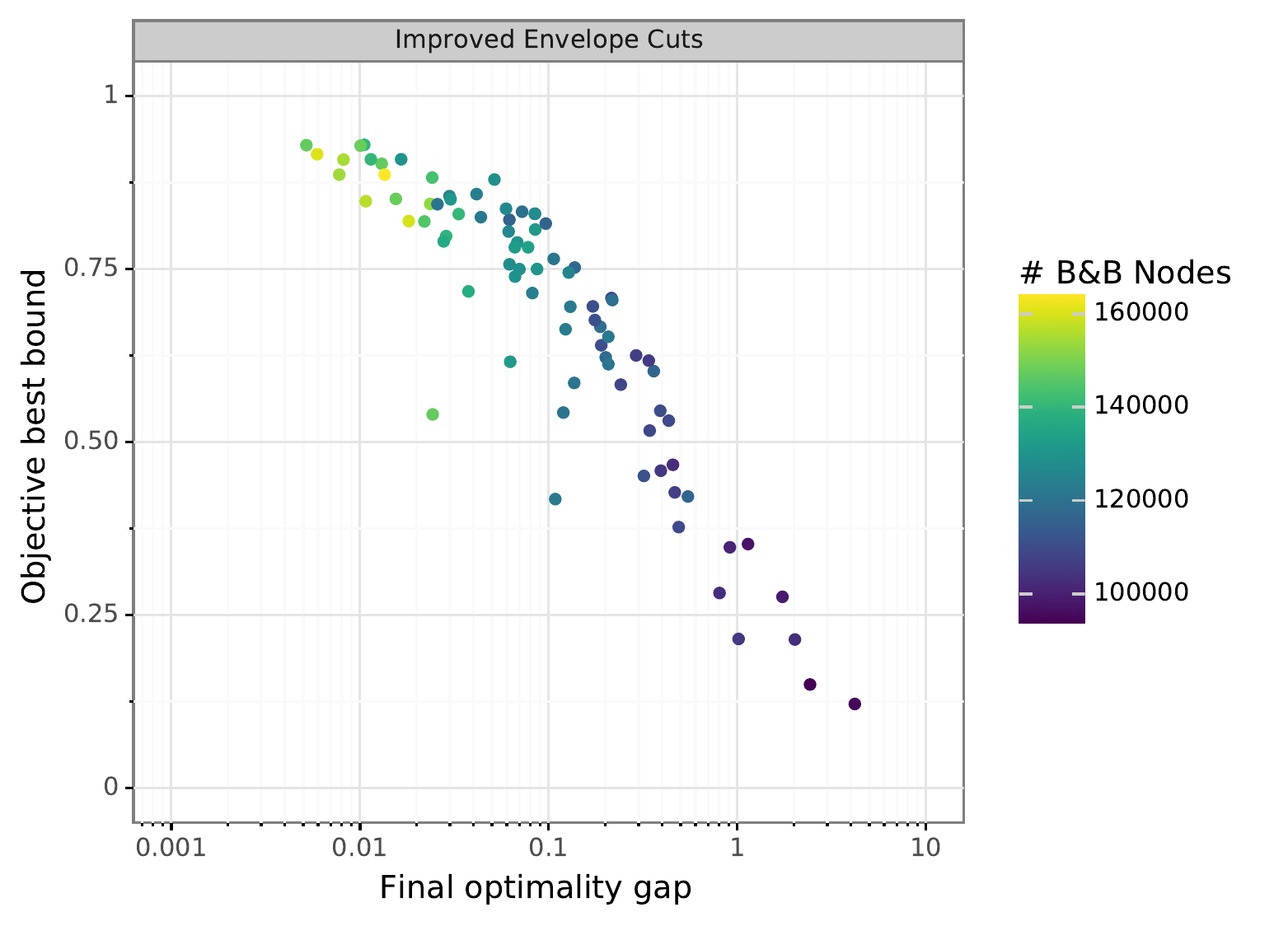}~%
	\includegraphics[width=.48\linewidth]{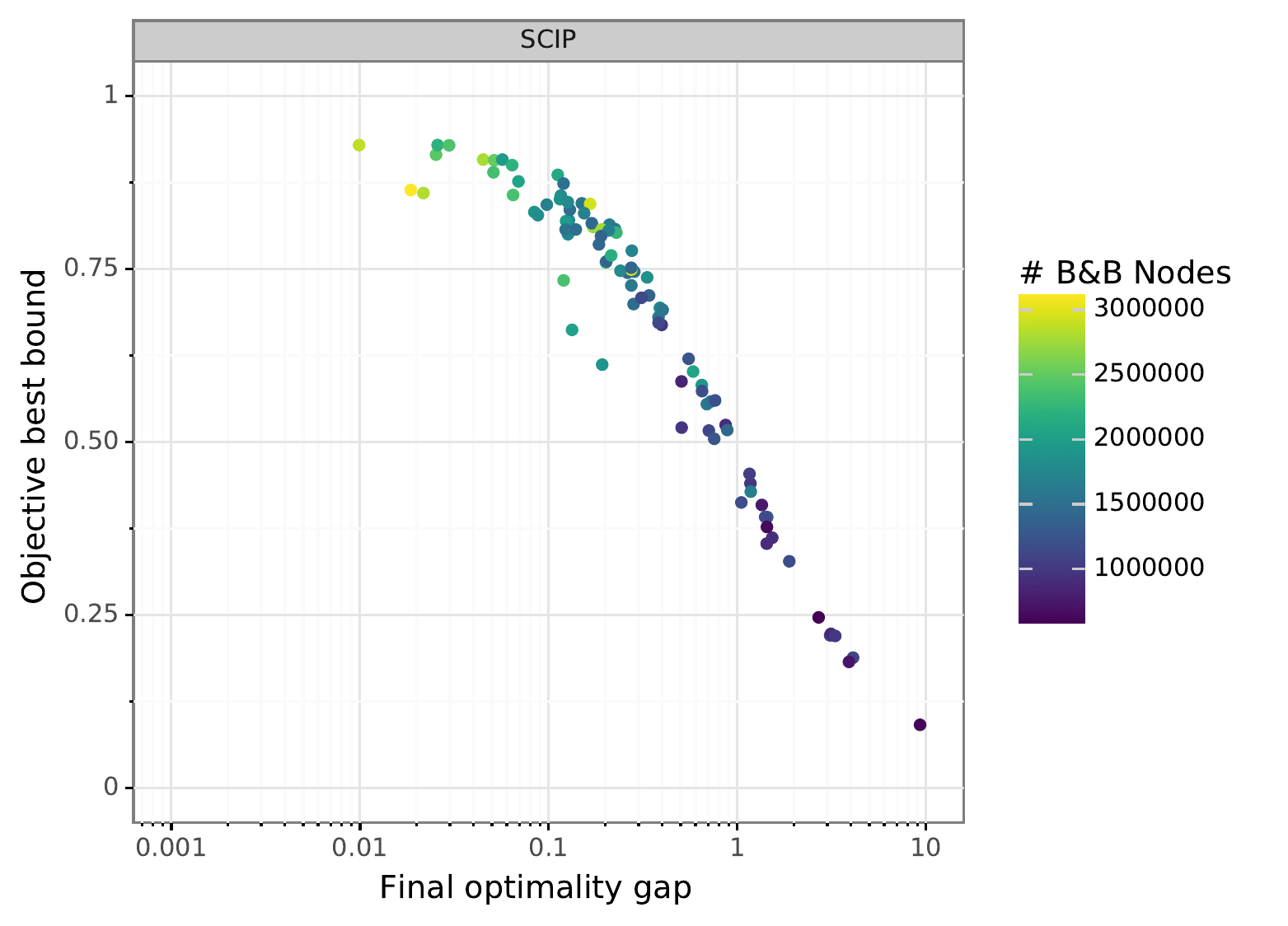}
	\caption{Final optimality gaps and number of nodes traversed in the branch-and-bound tree ($|E|=100$, $\alpha=0.6$)}\label{fig:resultNodes}
\end{figure}

Figure~\ref{fig:resultNodes} shows the optimality gaps (log scale) versus the best-bound objective value obtained by the improved envelope cuts setting and SCIP on the unsolved instances for the case of $100$ edges and $\alpha=0.6$. This figure indicates that the problems become harder when the reliability of the problem is lower: the optimality gaps are large when the objective bound is low. This can be explained because this is the region where the difference between the nonconvex functions and their concave envelopes differ the most (recall Figure~\ref{fig:results}), so the approximation is not sufficiently tight to lead the solver to prove the optimality of the solutions. This is also correlated with the number of nodes traversed in the branch-and-bound: the number of branch-and-bound nodes are smaller in harder instances, indicating that the subproblems at each node are harder to solve (probably because they include a larger number of cuts). Interestingly, similar behaviors also occur in SCIP, even if the latter is able to visit 10 times more nodes of the branch-and-bound tree.

%\todo[inline]{GM: En esto último tengo dudas. Por qué se dice que los problemas con bajo reliability es la región donde el concave envelope es más distinto a la función? No veo bien de donde sale esa conclusión. El comentario que sigue, donde se dice que se agregaron muchos cortes, me hace mucho sentido, pero esto no estaría diciendo nada sobre la calidad de los envelopes.}

% We observe that previous solutions are optimal for an approximation of the problem. To evaluate its global optimality, we found the true optimal value only for the case of $|E|=50$ edges using a brute-force algorithm. 
% Figure~\ref{fig:results2} shows the comparison of the true reliability of the resulting solutions versus the optimal reliability of the problem.  It is observed that the solutions including the concave envelope provides a very good approximation of the optimal value. In fact, in almost all cases, the relative error of the obtained solution is less than $1\%$ for $\alpha=60\%$ and less than $0.1\%$ for $\alpha=80\%$. It is also observed that by not adding these cuts the obtained solutions become worse, obtaining a relative error higher than $10\%$ for $\alpha=0.6$ in most  cases. 

\section{Conclusions and further extensions}\label{sec:conclusions} 

We provide an optimization framework to solve network design problems for maximizing the all-terminal reliability problem on series-parallel graphs when failure probabilities are independent but not identical. Our approach exploits the use of concave envelopes of the nonconcave functions that can be implemented successfully using current optimization solvers, something that has not been explored thus far in this context.

The special properties of the functions that appear in reliability optimization allow us to derive envelopes that can be refined and exploited in the solution process. Computational experiments show that it is highly beneficial to perform such refinements of the concave envelopes along the branch-and-bound process and thus provide better local approximations for the nonlinear functions. If this is not done, the solver faces difficulties in obtaining good solutions or proving optimality.  

These techniques can be extended to more general contexts of network reliability optimization. For example, similar ideas can be used for $K$-terminal reliabilities, where the functions associated to other reliability-preserving reductions (see \cite{Satyanarayana}) could also be approximated by their concave envelopes in a similar way. However, these concave envelopes are not known and can be difficult to find in closed form. Therefore, developing new techniques, such as those presented in~\cite{barrera2021convex}, that can handle these functions is a promising future direction that can considerably widen the applicability of our proposed framework.  Additionally, the use of convex/concave envelopes for reliability optimization can also be applied to more general families of graphs. In fact, the reductions discussed in this paper apply to any graph and allow us to reduce the size of the problem. The smaller problem can be solved using other optimization techniques such as sample average approximation. This appears very promising, in particular for graphs with small treewidth, as recently discussed in~\cite{goharshady2020efficient}.

\bibliographystyle{amsplain}      % mathematics and physical sciences
\bibliography{biblio}
\end{document}